\documentclass{amsart}

\usepackage{layout}
\usepackage{amsmath, amsthm}
\usepackage{amssymb, amsrefs}
\usepackage{stmaryrd}
\usepackage{graphicx}
\usepackage{setspace}
\usepackage{textcomp}
\usepackage{esvect}
\usepackage[letterpaper]{geometry}
\usepackage{arydshln}
\usepackage[all]{xy}
\usepackage{etex}

\usepackage[dvipsnames]{xcolor}

\usepackage{epsfig}
\usepackage{latexsym}
\usepackage{amsthm}

\usepackage{mathrsfs}

\usepackage[colorlinks,citecolor=blue]{hyperref}

\usepackage[latin1]{inputenc}

\usepackage{tikz-cd}
\usetikzlibrary{graphs,graphs.standard}
\usepackage{pgfplots}
\usetikzlibrary{matrix,arrows,decorations.pathmorphing}
\usetikzlibrary{shapes.geometric}

\newtheorem{theorem}{Theorem}[section]
\newtheorem{lemma}[theorem]{Lemma}
\newtheorem{proposition}[theorem]{Proposition}

\newtheorem{corollary}[theorem]{Corollary}

\theoremstyle{definition}
\newtheorem{definition}[theorem]{Definition}

\newtheorem{example}[theorem]{Example}

\newtheorem{remark}[theorem]{Remark}
\newtheorem{obs}[theorem]{Observation}

\theoremstyle{definition}
\numberwithin{equation}{section}

\tikzset{
    ncbar angle/.initial=90,
    ncbar/.style={
        to path=(\tikztostart)
        -- ($(\tikztostart)!#1!\pgfkeysvalueof{/tikz/ncbar angle}:(\tikztotarget)$)
        -- ($(\tikztotarget)!($(\tikztostart)!#1!\pgfkeysvalueof{/tikz/ncbar angle}:(\tikztotarget)$)!\pgfkeysvalueof{/tikz/ncbar angle}:(\tikztostart)$)
        -- (\tikztotarget)
    },
    ncbar/.default=0.5cm,
}

\tikzset{square left brace/.style={ncbar=0.5cm}}
\tikzset{square right brace/.style={ncbar=-0.5cm}}

\tikzset{round left paren/.style={ncbar=0.5cm,out=120,in=-120}}
\tikzset{round right paren/.style={ncbar=0.5cm,out=60,in=-60}}

\def \con {\sim}
\definecolor{laura}{rgb}{.4, 0, .6}

\begin{document}
\pagestyle{headings}

 \title{Fundamental groups of Hamming graphs}
    \author{K. Behal, T. Chih}

\maketitle

\vspace*{-2em}


\begin{abstract}
Recently there has been growing interest in discrete homotopies and homotopies of graphs beyond treating graphs as 1-dimensional simplicial spaces.  One such type of homotopy is $\times$-homotopy.  Recent work by Chih-Scull has developed a homotopy category, a fundamental group for graphs under this homotopy, and a way of computing covers of graphs that lift homotopy via this fundamental group.  In this paper, we compute the fundamental groups of all Hamming graphs, show that they are direct products of cyclic groups, and use this result to describe some $\times$-homotopy covers of Hamming graphs.

\end{abstract}

\section{Introduction}

In topology, the fundamental group of a space is a tool that is instrumental in classifying spaces up to homeomorphism.  Various fundamental groups have been developed for graphs using different generalizations. 
  This research follows work done by Chih and Scull \cite{CS3} where the generators of the group are based circuits and  $\times$-homotopy is the analogue of classical homotopy.  Chih-Scull defined the fundamental group for graphs under this notion, and computed several examples. $\times$-homotopy is a well studied notion of homotopy with connections to topology and applications in graph theory \cite{CS1, CS3, CS2, Docht1, Docht2}.

In particular, we compute the $\times$-homotopy fundamental groups of all Hamming graphs. We start by continuing the investigation of pleats (or stiff representatives) that Chih-Scull conducted. They showed that since fundamental groups are homotopy invariant, we can compute a fundamental group for any graph by computing the fundamental group of its pleat. Thus the fundamental groups of stiff graphs are of particular importance.  We show here that Hamming graphs with elements or dimension greater than 2 are pleats.  We show that the generators of a Hamming graph are the generators of the fundamental groups of the complete subgraphs associated with each dimension.  We also show that these generators commute and are distinct, thus showing that the fundamental group of a Hamming graph is a direct product of cyclic groups.

Our paper is organized in the following manner. In Section \ref{S:prelim}, we define properties of graphs and provide information about notation. In Section \ref{S:H(d,2)}, we discuss whether Hamming graphs are pleats and investigate the special properties of Hamming graphs with 2 elements ($q=2$). In Section \ref{H(d,q)}, we explore properties of Hamming graphs more generally, compute the fundamental groups for all Hamming graphs, and use the fundamental group to describe some covers.

\section{Preliminaries}\label{S:prelim}

In this section we establish some preliminary facts and definitions that are used throughout the paper. Standard graph theory terminology and notation can be found in \cite{BM} and group theory terminology and 
notation can be found in \cite{gallian2016contemporary}.  We note that in this paper, we will allow single loops, but not multiple edges between vertices.

We begin with the definition of a graph, and in particular a Hamming graph, as the  focus of this paper.

\begin{definition}[\cite{BM, AGT}]
    A \textbf{graph} is a pair $G = (V, E)$, where V is a set of vertices (singular: vertex), and E is a set of  one or two sets of \textbf{vertices} which we call \textbf{edges}. We denote the vertex set and edge set of a graph $G$ as $V(G)$ and $E(G)$ respectively. We denote the existence of an edge between two vertices ${v, u} \in E(G)$ by writing $v \sim u$.  In this case that $u\sim u$,we say that $u$ is \textbf{looped}.

\end{definition}

\begin{definition}
    Let $d$ and $q$ be positive integers, and let $S$ be a set with $q$ elements. The \textbf{Hamming graph} is the graph with vertex set $V(H(d,q)) = \{(x_1, \ldots, x_d) : x_i\in  S\}$ where two vertices are adjacent $(x_1, \ldots, x_d) \sim (y_1, \ldots, y_d)$ if they differ in exactly one coordinate. 
\end{definition}

Notice that this notation for vertices can be cumbersome to read and write, and so we denote $(x_1, \ldots, x_i, \ldots, x_d)$ as $(x_1,-,x_i,-,x_d)$ or $(-,x_i,-)$ if it is understood what the entries for the tuples are.  Similarly, we will use $\mathbf{0}$ to denote the all $0$'s tuple.

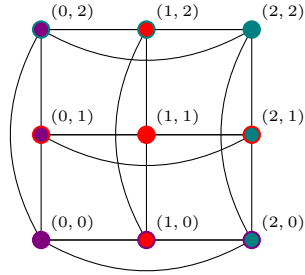
\begin{figure}[h]\label{F:h23}

\[\begin{tikzpicture}[scale=1.4]
\foreach \x in {0,...,2}{
\foreach \y in {0,...,2}{

\draw (\x,\y) --node[above right]{\tiny $(\x, \y)$} (\x, \y);

}
};

\foreach \x in {0,...,2}{
\draw (\x, 0)--(\x, 2);
\draw (\x,0)edge[bend left](\x,2);
};

\draw (0,0)--(2,0);
\draw (0,2)--(2,2);
\draw (0,1)--(2,1);

\foreach \y in {0,...,2}{
\draw (0, \y)--(2, \y);
\draw (0,\y)edge[bend right](2,\y);
};

\draw[fill, violet] (0,0) circle (2pt);
\draw[fill, red] (1,0) circle (2pt);
\draw[fill, teal] (2,0) circle (2pt);
\draw[fill, violet] (0,1) circle (2pt);
\draw[fill, red] (1,1) circle (2pt);
\draw[fill, teal] (2,1) circle (2pt);
\draw[fill, violet] (0,2) circle (2pt);
\draw[fill, red] (1,2) circle (2pt);
\draw[fill, teal] (2,2) circle (2pt);

\draw[ thick, violet] (0,0) circle (2.2pt);
\draw[thick,  red] (0,1) circle (2.2pt);
\draw[ thick,  teal] (0,2) circle (2.2pt);
\draw[thick,  violet] (1,0) circle (2.2pt);
\draw[thick,  red] (1,1) circle (2.2pt);
\draw[thick,  teal] (1,2) circle (2.2pt);
\draw[thick,  violet] (2,0) circle (2.2pt);
\draw[thick,  red] (2,1) circle (2.2pt);
\draw[thick,  teal] (2,2) circle (2.2pt);
 \end{tikzpicture}\]

\caption{The Hamming graphs $H(2,3)$.}

\end{figure}

Also  note that for any Hamming graph, if we pick an index $i$ and fix the entries for every other index, the $d$ vertices we obtain form a complete graph.  In particular, the subgraphs induced by $\{(0,-,x,-,0): x\in [d]\}$ are complete graphs.  In Figure 1, fixing $(a,-)$ or $(-,b)$ the resulting trio of vertices form a complete graph.

Next, we establish some preliminaries of $\times$-homotopy.

\begin{definition}
        A \textbf{homomorphism} of graphs $f : G \to H$ is given by a set map $f : V (G) \mapsto V (H)$ such that if $v_1 \sim v_2 \in E(G)$ then $f(v_1) \sim f(v_2) \in E(H)$.

\end{definition}

\begin{definition}[\cite{CS1}]
    Let $f, g : G \to H$ be graph homomorphisms. We say $f$ and $g$ are a \textbf{spider pair} if there exists a single vertex $v$ such that $f(u) = g(u)$ for all $u \neq v$ and $f(v) \neq g(v)$. When we replace $f$ with $g$ we refer to it as a spider move. 
\end{definition}

Since fundamental groups are homotopy classes of closed walks (circuits), we define some of the basic definitions used in this construction.

\begin{definition}[\cite{CS2}]
    Let $\alpha = (v_0v_1v_2\cdots v_n)$ be a walk in G. We say that $\alpha$ is \textbf{prunable} if $v_i = v_{i+2}$ for some $i$. We define a \textbf{prune} of $\alpha$ to be given by a walk $\alpha'$ obtained by deleting the vertices $v_i$ and $v_{i+1}$ from the walk when $v_i = v_{i+2}$: 
    if $\alpha = (v_0v_1v_2\cdots v_{i-1}v_iv_{i+1}v_iv_{i+3}\cdots v_n)$, then the prune of $\alpha$ is $\alpha' = (v_0v_1v_2\cdots v_{i-1}v_iv_{i+3}\cdots v_n)$.

\end{definition}

\begin{definition}
    Given graphs $G, H$, we define the \textbf{categorical product} denoted $G\times H$ to be a graph where: 
    \begin{itemize}
    \item  A vertex is a  pair $(v, w)$ where $v \in V(G)$ and $w \in V(H)$.
\item An edge is defined by $(v_1, w_1) \con (v_2, w_2) \in E(G \times H) $ for $v_1 \con v_2 \in E(G)$ and $w_1 \con w_2 \in E(H)$.   
    \end{itemize}
\end{definition}

\begin{example}\label{E:examprod}

Let $G$ be the graph on two adjacent looped vertices:  $V(G) = \{ 0, 1\}$ and $E(G) = \{ 0\con 0,1 \con 1, 0 \con 1\}$.  Let  $H= K_2 $ with $V(H) = \{ a, b\} $ and $E(H) = \{ a\con b\}$.   Then $G \times H$ is isomorphic to the cyclic graph $C_4$:

$$\begin{tikzpicture}

\draw[fill] (1,.5) circle (2pt);
\draw (1,.5) --node[below]{$a$} (1,.5);
\draw[fill] (2,.5) circle (2pt);
\draw (2,.5) --node[below]{$b$} (2,.5);
\draw (1,.5) -- (2,.5);

\draw[fill] (.5,1) circle (2pt);
\draw (.5,1) --node[left]{0} (.5,1);
\draw[fill] (.5,2) circle (2pt);
\draw (.5,2) --node[left]{1} (.5,2);
\draw (.5,1) -- (.5,2);
\draw (.5,1)  to[in=-40,out=220,loop, distance=.7cm] (.5,1);
\draw (.5,2)  to[in=-220,out=40,loop, distance=.7cm] (.5,2);

\draw[fill] (1,1) circle (2pt);
\draw (1,1) --node[below]{\tiny $(0,a)$} (1,1);
\draw[fill] (1,2) circle (2pt);
\draw (2,1) --node[below]{\tiny $(0,b)$} (2,1);
\draw[fill] (2,1) circle (2pt);
\draw (1,2) --node[above]{\tiny $(1,a)$} (1,2);
\draw[fill] (2,2) circle (2pt);
\draw (2,2) --node[above]{\tiny $(1,b)$} (2,2);

\draw (1,1) -- (2,2);
  \draw (1,1) -- (2,1);
\draw (1,2) -- (2,1);
\draw (1,2) -- (2,2);

\end{tikzpicture}$$

\end{example}

\begin{definition}[\cite{CS1}]
    Given $f, g : G \to H$, we say that $f$ is $\times$-homotopic to $g$, written $f \simeq g$, if there is a graph homomorphism $\Lambda : G \times I_n^\ell \to H$ such that $\Lambda|_{G\times \{0\}} = f$ and $\Lambda|_{G\times \{n\}} = g$, where $I_n^\ell$ denotes a looped path graph.
\end{definition}

Since $\times$-homotopy is the only form of homotopy discussed in this paper, we will use $\times$-homotopy and homotopy interchangeably.

\begin{definition}[\cite{CS1}]
    Suppose that $P_n$ denotes the path graph with $n$ vertices and
$\alpha   : P_{n} \to G, \alpha, \beta   : P_n \to G$ are walks in $G$ from $x$ to $y$, so $\alpha(0) = \beta  (0) = x$ and $\alpha(n) = \beta  (n) = y$. We say $\alpha$ and
$\beta  $ are homotopic rel endpoints if $\alpha  \simeq   \beta  $ in such a way that all intermediate walks $\lambda     |_{G\times  \{i\}}$ are
also walks from $x$ to $y$, so the endpoints of the walk remain fixed: $\lambda     |_{\{0\}\times  I_m} = x$ and $\lambda     |_{\{n\}\times  I_m} = y$.
\end{definition}

\begin{remark}

The above definition is defined when two maps from $P_n$ to $G$ are homotopic rel-endpoints. In \cite{CS2} this definition is expanded and it is shown that two walks $\alpha, \beta$ are homotopic rel endpoints if we can transform $\alpha$ to $\beta$  through a series of spider moves, prunes and anti-prunes.  The definition and results are technical in nature and thus not included here for ease of reading. Throughout this paper, we refer to such walks as \textbf{equivalent}.

\end{remark}

\begin{example}\label{E:prune}

 Let $\alpha={(acbce)}$ and  $\beta={(ade)}$ be walks in the graph below. Then  $\alpha \simeq \beta$  since we have a prune of   $\alpha$ to $\alpha' = (ace)$ and then a spider move to  $\beta= (ade)$.

$$\begin{tikzpicture}
\draw[fill] (0,0) circle (2pt);
\draw (0,0) --node[left]{$d$} (0,0);
\draw[fill] (.5,.707) circle (2pt);
\draw (.5,.707) --node[above]{$a$} (.5,.707);
\draw[fill] (.5,-.707) circle (2pt);
\draw (.5,-.707) --node[below]{$e$} (.5,-.707);
\draw[fill] (1,0) circle (2pt);
\draw (1,0) --node[below]{$c$} (1,0);
\draw[fill] (2, 0) circle (2pt);
\draw (2,0) --node[above]{$b$} (2,0);

\draw (0,0) -- (.5,.707) -- (1,0) -- (.5,-.707)--(0,0);
\draw (1,0)--(2,0);

\draw[ultra thick, orange] (.5, .707) -- (1, 0) -- (2, 0) -- (1,0)--(.5, -.707);

\end{tikzpicture}
\ \ \ \ 
\begin{tikzpicture}
\draw[fill] (0,0) circle (2pt);
\draw (0,0) --node[left]{$d$} (0,0);
\draw[fill] (.5,.707) circle (2pt);
\draw (.5,.707) --node[above]{$a$} (.5,.707);
\draw[fill] (.5,-.707) circle (2pt);
\draw (.5,-.707) --node[below]{$e$} (.5,-.707);
\draw[fill] (1,0) circle (2pt);
\draw (1,0) --node[below]{$c$} (1,0);
\draw[fill] (2, 0) circle (2pt);
\draw (2,0) --node[above]{$b$} (2,0);

\draw (0,0) -- (.5,.707) -- (1,0) -- (.5,-.707)--(0,0);
\draw (1,0)--(2,0);

\draw[ultra thick, orange] (.5, .707) -- (1, 0) --(.5, -.707);

\end{tikzpicture}
\ \ \ \ 
\begin{tikzpicture}
\draw[fill] (0,0) circle (2pt);
\draw (0,0) --node[left]{$d$} (0,0);
\draw[fill] (.5,.707) circle (2pt);
\draw (.5,.707) --node[above]{$a$} (.5,.707);
\draw[fill] (.5,-.707) circle (2pt);
\draw (.5,-.707) --node[below]{$e$} (.5,-.707);
\draw[fill] (1,0) circle (2pt);
\draw (1,0) --node[below]{$c$} (1,0);
\draw[fill] (2, 0) circle (2pt);
\draw (2,0) --node[above]{$b$} (2,0);

\draw (0,0) -- (.5,.707) -- (1,0) -- (.5,-.707)--(0,0);
\draw (1,0)--(2,0);

\draw[ultra thick, orange] (.5, .707) -- (0,0) -- (.5, -.707);

\end{tikzpicture}
$$

\end{example}

\begin{definition}
    Given a graph $G$ and $v\in G$, the \textbf{neighborhood} of $v$ in $G$ denoted $N(v)$ is the set of all vertices adjacent to $v$.
\end{definition}

\begin{definition}[\cite{GMDG, BonatoCaR, CS1}] \label{D:nerve}  We say that a graph $G$ is {\bf stiff} or a \textbf{pleat}  if there are no two distinct vertices $v, w$ such that $N(v) \subseteq N(w)$. 

\end{definition}

Finally we include some definitions and results about covers of graphs and their relation to fundamental groups.

\begin{definition}[\cite{Angluin1}]\label{D:Cover}
A \textbf{covering map} is a  graph morphism $f:\widetilde{G}\to G$ such that: 
\begin{itemize} 
\item $f$ is a surjection on vertices and \item  given any $v \in V(G)$  and $\tilde{v}\in f^{-1}(v)$,  $f$ induces a bijection  $N(\tilde{v})\to N(v)$,  \end{itemize}
\end{definition}

The notion of covers in graphs as defined by Angluin in \cite{Angluin1} is extended in \cite{CS3} to describe covers which allow $\times$-homotopy of walks, and of $\times$-homotopy more generally, to lift as in  classical algebraic topology.  In order to achieve this, we introduce an extension  of the two-neighborhood, which is a set of walks, rather than vertices.

\begin{definition}[\cite{CS3}]\label{D:2Nbd}
For any vertex $v \in V(G)$ we define the extended neighborhood $N_2(v)$ to be the walks of length $2$ starting at $v$.  
\end{definition}

Our covering  maps induce bijections on these two-neighborhoods.  

\begin{definition}[\cite{CS3}]\label{D:hocover}

A graph morphism between   graphs  $f:\widetilde{G}\to G$ is a \textbf{homotopy covering map} if given any $v\in V(G)$ and $\tilde{v}\in V(\widetilde{G})$ such that ${f}(\tilde{v})=v$, then $f$ induces a bijection $N_2(\tilde{v}) \to N_2 (v)$ and this bijection respects endpoints in the sense that walks in $N_2(\tilde{v})$ have the same endpoint if and only if their corresponding walks in $N_2(v)$ do also.  
\end{definition}

\begin{obs}  \label{O:HcoverisCover}
If $f:  \widetilde{G} \to G$ is a homotopy covering map  then it is also a covering map in the sense of Definition \ref{D:Cover}, since the bijection of length two walks requires a bijection of  walks of the form   $(vwv)$ for $w \in N(v)$.   
\end{obs}

\begin{example}\label{Example:HCover} 
The following is a homotopy covering map: 

$$
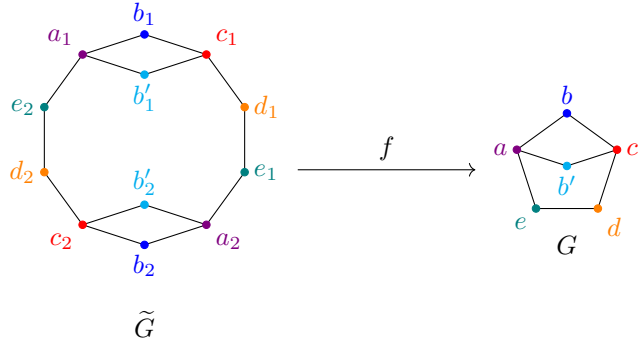
\begin{tikzpicture}[scale=.7]
\node(G) at (0,0){\begin{tikzpicture}[scale=.7]
\draw (0,1)\foreach \x in {1,...,5}{--({sin(72*\x)}, {cos(72*\x)})};
\draw (-.951, 0.309)--(0,0)--(.951,0.309);

\draw[blue] ({sin(72*0)}, {cos(72*0)}) -- node[above]{$b$} ({sin(72*0)}, {cos(72*0)});
\draw[red] ({sin(72*1)}, {cos(72*1)}) -- node[right]{$c$} ({sin(72*1)}, {cos(72*1)});
\draw[orange] ({sin(72*2)}, {cos(72*2)}) -- node[below right]{$d$} ({sin(72*2)}, {cos(72*2)});
\draw[teal] ({sin(72*3)}, {cos(72*3)}) -- node[below left]{$e$} ({sin(72*3)}, {cos(72*3)});
\draw[violet] ({sin(72*4)}, {cos(72*4)}) -- node[left]{$a$} ({sin(72*4)}, {cos(72*4)});
\draw[cyan] (0, 0) -- node[below]{$b'$} (0, 0);

\draw[fill, blue] ({sin(72*0)}, {cos(72*0)}) circle (2pt);
\draw[fill, red] ({sin(72*1)}, {cos(72*1)}) circle (2pt);
\draw[fill, orange] ({sin(72*2)}, {cos(72*2)}) circle (2pt);
\draw[fill, teal] ({sin(72*3)}, {cos(72*3)}) circle (2pt);
\draw[fill, violet] ({sin(72*4)}, {cos(72*4)}) circle (2pt);
\draw[fill, cyan] (0, 0) circle (2pt);

\node at (0, -1.5){$G$};

\end{tikzpicture}};

\node(T) at (-8,0){\begin{tikzpicture}[scale=0.7]
\draw (0,2)\foreach \x in {1,...,10}{--({2*sin(36*\x)}, {2*cos(36*\x)})};
\draw (-1.176, 1.618)--(0,1.236)--(1.176,1.618);
\draw (-1.176, -1.618)--(0,-1.236)--(1.176,-1.618);

\draw[blue] ({2*sin(36*0)}, {2*cos(36*0)}) -- node[above]{$b_1$} ({2*sin(36*0)}, {2*cos(36*0)});
\draw[red] ({2*sin(36*1)}, {2*cos(36*1)}) -- node[above right]{$c_1$} ({2*sin(36*1)}, {2*cos(36*1)});
\draw[orange] ({2*sin(36*2)}, {2*cos(36*2)}) -- node[right]{$d_1$} ({2*sin(36*2)}, {2*cos(36*2)});
\draw[teal] ({2*sin(36*3)}, {2*cos(36*3)}) -- node[right]{$e_1$} ({2*sin(36*3)}, {2*cos(36*3)});
\draw[violet] ({2*sin(36*4)}, {2*cos(36*4)}) -- node[below right]{$a_2$} ({2*sin(36*4)}, {2*cos(36*4)});
\draw[blue] ({2*sin(36*5)}, {2*cos(36*5)}) -- node[below]{$b_2$} ({2*sin(36*5)}, {2*cos(36*5)});
\draw[red] ({2*sin(36*6)}, {2*cos(36*6)}) -- node[below left]{$c_2$} ({2*sin(36*6)}, {2*cos(36*6)});
\draw[orange] ({2*sin(36*7)}, {2*cos(36*7)}) -- node[left]{$d_2$} ({2*sin(36*7)}, {2*cos(36*7)});
\draw[teal] ({2*sin(36*8)}, {2*cos(36*8)}) -- node[left]{$e_2$} ({2*sin(36*8)}, {2*cos(36*8)});
\draw[violet] ({2*sin(36*9)}, {2*cos(36*9)}) -- node[above left]{$a_1$} ({2*sin(36*9)}, {2*cos(36*9)});
\draw[cyan] (0,1.236) -- node[below]{$b'_1$}(0,1.236);
\draw[cyan] (0,-1.236) -- node[above]{$b'_2$}(0,-1.236);

\foreach \x in {0,...,1}{\draw[fill, blue] ({2*sin(180*\x)}, {2*cos(180*\x)}) circle (2pt);}
\foreach \x in {0,...,1}{\draw[fill, red] ({2*sin(180*\x+36)}, {2*cos(180*\x+36)}) circle (2pt);}
\foreach \x in {0,...,1}{\draw[fill, orange] ({2*sin(180*\x+2*36)}, {2*cos(180*\x+2*36)}) circle (2pt);}
\foreach \x in {0,...,1}{\draw[fill, teal] ({2*sin(180*\x+3*36)}, {2*cos(180*\x+3*36)}) circle (2pt);}
\foreach \x in {0,...,1}{\draw[fill, violet] ({2*sin(180*\x+4*36)}, {2*cos(180*\x+4*36)}) circle (2pt);}
\draw[fill, cyan] (0,1.236) circle (2pt);
\draw[fill, cyan] (0,-1.236) circle (2pt);

\node at (0, -3.5){$\widetilde{G}$};

\end{tikzpicture}};

\draw[->] (T) --node[above]{$f$} (G);

\end{tikzpicture}
$$

Where the covering map $f:\tilde{G}\to G$ maps $x_i\mapsto x$.  One can verify that this is a group homomorphism. Here we see that the diamond $abcb'$ lifts to a diamond in $\widetilde{G}$, and thus the two-paths also lift in a way that respects endpoints.

Finally, we introduce the notion of a \textbf{fundamental group}, which in this context consists of $\times$-homotopy classes of circuits.  Classifying the fundamental group for Hamming graphs will the the main goal of this paper.  We also introduce the notion of a universal cover of a graph.  In \cite{CS3}, it is shown that quotienting the universal cover by subgroups of the fundamental group of a graph produces homotopy covers, as in classical algebraic topology.

\end{example}

\begin{definition}[\cite{CS2}]\label{D:FundGroup}
    Given a connected graph $G$ and vertex $v\in V(G)$, the \textbf{fundamental group} of $G$ denoted $\Pi(G)$ is the collection of $\times$-homotopy equivalence classes of closed walks or circuits starting and ending at $v$, with concatenation as the operation.
\end{definition}

\begin{example}
    Consider the graph $G$ from Example \ref{Example:HCover}:
    \[\begin{tikzpicture}[scale=.7]
\draw (0,1)\foreach \x in {1,...,5}{--({sin(72*\x)}, {cos(72*\x)})};
\draw (-.951, 0.309)--(0,0)--(.951,0.309);

\draw[blue] ({sin(72*0)}, {cos(72*0)}) -- node[above]{$b$} ({sin(72*0)}, {cos(72*0)});
\draw[red] ({sin(72*1)}, {cos(72*1)}) -- node[right]{$c$} ({sin(72*1)}, {cos(72*1)});
\draw[orange] ({sin(72*2)}, {cos(72*2)}) -- node[below right]{$d$} ({sin(72*2)}, {cos(72*2)});
\draw[teal] ({sin(72*3)}, {cos(72*3)}) -- node[below left]{$e$} ({sin(72*3)}, {cos(72*3)});
\draw[violet] ({sin(72*4)}, {cos(72*4)}) -- node[left]{$a$} ({sin(72*4)}, {cos(72*4)});
\draw[cyan] (0, 0) -- node[below]{$b'$} (0, 0);

\draw[fill, blue] ({sin(72*0)}, {cos(72*0)}) circle (2pt);
\draw[fill, red] ({sin(72*1)}, {cos(72*1)}) circle (2pt);
\draw[fill, orange] ({sin(72*2)}, {cos(72*2)}) circle (2pt);
\draw[fill, teal] ({sin(72*3)}, {cos(72*3)}) circle (2pt);
\draw[fill, violet] ({sin(72*4)}, {cos(72*4)}) circle (2pt);
\draw[fill, cyan] (0, 0) circle (2pt);

\node at (0, -1.5){$G$};

\end{tikzpicture}\]

The fundamental group $G$ consists of classes of walks starting and ending at $a$.  Note that the walks $(abcdea)$, $(ab'cdea)$ and $(abcb'abcdea)$ are equivalent via spider moves and prunes.  It is straightforward to see that $(abcdea)$ generates the fundamental group, and so $\Pi(G)\cong \mathbb{Z}$.

\end{example}

\begin{definition}[\cite{CS3}]\label{D:UC}
Given a graph $G$ the \textbf{universal cover} of $G$ is the graph $U$ up to isomorphism such that if $\rho:\tilde{G}\to G$ is a covering and   $\tilde{v}\in f^{-1}(v)$.  Then there is a unique map  $\tilde{\rho}: U\to \widetilde{G}$ such that $\rho=f\circ\tilde{\rho}$ and $\tilde{\rho}([v])=\tilde{v}$, and the map $\tilde{\rho}$ is a homotopy cover.

\end{definition}

\begin{example}\label{Example:UCover}
    The universal cover $U$ of $G$ from Example \ref{Example:HCover} may be depicted as follows: $$\begin{tikzpicture}[scale=0.8]

\draw (-7.5,0)--(-5,0)--(-4,.5)--(-3,0)--(0,0)--(1,.5)--(2,0)--(5,0)--(6,.5)--(7,0)--(7.5,0);
\draw (-5,0)--(-4,-.5)--(-3,0);
\draw (0,0)--(1,-.5)--(2,0);
\draw (5,0)--(6,-.5)--(7,0);
\node at (-8,0){$\cdots$};
\node at (8,0){$\cdots$};

\foreach \x in {-1,...,1}{ \draw[violet, fill] (\x*5, 0) circle (2pt);}
\foreach \x in {-1,...,1}{ \draw[blue, fill] (\x*5+1, 0.5) circle (2pt);}
\foreach \x in {-1,...,1}{ \draw[cyan, fill] (\x*5+1, -0.5) circle (2pt);}
\foreach \x in {-1,...,1}{ \draw[red, fill] (\x*5+2, 0) circle (2pt);}
\foreach \x in {-1,...,1}{ \draw[teal, fill] (\x*5-1, 0) circle (2pt);}
\foreach \x in {-1,...,1}{ \draw[orange, fill] (\x*5-2, 0) circle (2pt);}

\draw (0,0)--node[below]{\tiny $a$} (0,0);
\draw (1,0.5)--node[above]{\tiny $ab$} (1,0.5);
\draw (1,-0.5)--node[below]{\tiny $ab'$} (1,-0.5);
\draw (2,0)--node[below]{\tiny $a\overline{b}c$} (2,0);
\draw (3,0)--node[below]{\tiny $a\overline{b}cd$} (3,0);
\draw (4,0)--node[below]{\tiny $a\overline{b}cde$} (4,0);
\draw (5,0)--node[below]{\tiny $a\overline{b}cdea$} (5,0);
\draw (6,0.5)--node[above]{\tiny $a\overline{b}cdeab$} (6,0.5);
\draw (6,-0.5)--node[below]{\tiny $a\overline{b}cdeab'$} (6,-0.5);
\draw (7,0)--node[below right]{\tiny $a\overline{b}cdea\overline{b}c$} (7,0);

\draw (-1,0) --node[below]{\tiny $ae$} (-1,0);
\draw (-2,0) --node[below]{\tiny $aed$} (-2,0);
\draw (-3,0) --node[below]{\tiny $aedc$} (-3,0);
\draw (-4,0.5) --node[above]{\tiny $aedcb$} (-4,0.5);
\draw (-4,-0.5) --node[below]{\tiny $aedcb'$} (-4,-0.5);
\draw (-5,0) --node[below]{\tiny $aedc\overline{b}a$} (-5,0);
\draw (-6,0) --node[above]{\tiny $aedc\overline{b}ae$} (-6,0);
\draw (-7,0) --node[below]{\tiny $aedc\overline{b}aed$} (-7,0);

\node at (0,-1.5){$U$};

\end{tikzpicture}$$

The vertices are homotopy classes of walks starting at $a$, noting that $abc\simeq ab'c$, and two walks are adjacent if the differ by one step.  We note that $\rho:U\to G$ maps $(a\cdots x)\mapsto x$.

One may likewise observe that if we consider the cover $\tilde{G}$ from Example \ref{Example:HCover}, that there is a homotopy covering map $\tilde{\rho}:U\to \tilde{G}$ so that $\rho = f\circ \tilde{\rho}$.

$$
\begin{tikzpicture}
\node(G) at (3,0){\begin{tikzpicture}
\draw (0,1)\foreach \x in {1,...,5}{--({sin(72*\x)}, {cos(72*\x)})};
\draw (-.951, 0.309)--(0,0)--(.951,0.309);

\draw[blue] ({sin(72*0)}, {cos(72*0)}) -- node[above]{$b$} ({sin(72*0)}, {cos(72*0)});
\draw[red] ({sin(72*1)}, {cos(72*1)}) -- node[right]{$c$} ({sin(72*1)}, {cos(72*1)});
\draw[orange] ({sin(72*2)}, {cos(72*2)}) -- node[below right]{$d$} ({sin(72*2)}, {cos(72*2)});
\draw[teal] ({sin(72*3)}, {cos(72*3)}) -- node[below left]{$e$} ({sin(72*3)}, {cos(72*3)});
\draw[violet] ({sin(72*4)}, {cos(72*4)}) -- node[left]{$a$} ({sin(72*4)}, {cos(72*4)});
\draw[cyan] (0, 0) -- node[below]{$b'$} (0, 0);

\draw[fill, blue] ({sin(72*0)}, {cos(72*0)}) circle (2pt);
\draw[fill, red] ({sin(72*1)}, {cos(72*1)}) circle (2pt);
\draw[fill, orange] ({sin(72*2)}, {cos(72*2)}) circle (2pt);
\draw[fill, teal] ({sin(72*3)}, {cos(72*3)}) circle (2pt);
\draw[fill, violet] ({sin(72*4)}, {cos(72*4)}) circle (2pt);
\draw[fill, cyan] (0, 0) circle (2pt);

\node at (0, -1.5){$G$};

\end{tikzpicture}};

\node(T) at (-3,0){\begin{tikzpicture}[scale=0.7]
\draw (0,2)\foreach \x in {1,...,10}{--({2*sin(36*\x)}, {2*cos(36*\x)})};
\draw (-1.176, 1.618)--(0,1.236)--(1.176,1.618);
\draw (-1.176, -1.618)--(0,-1.236)--(1.176,-1.618);

\draw[blue] ({2*sin(36*0)}, {2*cos(36*0)}) -- node[above]{$b_1$} ({2*sin(36*0)}, {2*cos(36*0)});
\draw[red] ({2*sin(36*1)}, {2*cos(36*1)}) -- node[above right]{$c_1$} ({2*sin(36*1)}, {2*cos(36*1)});
\draw[orange] ({2*sin(36*2)}, {2*cos(36*2)}) -- node[right]{$d_1$} ({2*sin(36*2)}, {2*cos(36*2)});
\draw[teal] ({2*sin(36*3)}, {2*cos(36*3)}) -- node[right]{$e_1$} ({2*sin(36*3)}, {2*cos(36*3)});
\draw[violet] ({2*sin(36*4)}, {2*cos(36*4)}) -- node[below right]{$a_2$} ({2*sin(36*4)}, {2*cos(36*4)});
\draw[blue] ({2*sin(36*5)}, {2*cos(36*5)}) -- node[below]{$b_2$} ({2*sin(36*5)}, {2*cos(36*5)});
\draw[red] ({2*sin(36*6)}, {2*cos(36*6)}) -- node[below left]{$c_2$} ({2*sin(36*6)}, {2*cos(36*6)});
\draw[orange] ({2*sin(36*7)}, {2*cos(36*7)}) -- node[left]{$d_2$} ({2*sin(36*7)}, {2*cos(36*7)});
\draw[teal] ({2*sin(36*8)}, {2*cos(36*8)}) -- node[left]{$e_2$} ({2*sin(36*8)}, {2*cos(36*8)});
\draw[violet] ({2*sin(36*9)}, {2*cos(36*9)}) -- node[above left]{$a_1$} ({2*sin(36*9)}, {2*cos(36*9)});
\draw[cyan] (0,1.236) -- node[below]{$b'_1$}(0,1.236);
\draw[cyan] (0,-1.236) -- node[above]{$b'_2$}(0,-1.236);

\foreach \x in {0,...,1}{\draw[fill, blue] ({2*sin(180*\x)}, {2*cos(180*\x)}) circle (2pt);}
\foreach \x in {0,...,1}{\draw[fill, red] ({2*sin(180*\x+36)}, {2*cos(180*\x+36)}) circle (2pt);}
\foreach \x in {0,...,1}{\draw[fill, orange] ({2*sin(180*\x+2*36)}, {2*cos(180*\x+2*36)}) circle (2pt);}
\foreach \x in {0,...,1}{\draw[fill, teal] ({2*sin(180*\x+3*36)}, {2*cos(180*\x+3*36)}) circle (2pt);}
\foreach \x in {0,...,1}{\draw[fill, violet] ({2*sin(180*\x+4*36)}, {2*cos(180*\x+4*36)}) circle (2pt);}
\draw[fill, cyan] (0,1.236) circle (2pt);
\draw[fill, cyan] (0,-1.236) circle (2pt);

\node at (0, -3.5){$\widetilde{G}$};

\end{tikzpicture}};

\node(U) at (0,5){\begin{tikzpicture}[scale=0.8]

\draw (-7.5,0)--(-5,0)--(-4,.5)--(-3,0)--(0,0)--(1,.5)--(2,0)--(5,0)--(6,.5)--(7,0)--(7.5,0);
\draw (-5,0)--(-4,-.5)--(-3,0);
\draw (0,0)--(1,-.5)--(2,0);
\draw (5,0)--(6,-.5)--(7,0);
\node at (-8,0){$\cdots$};
\node at (8,0){$\cdots$};

\foreach \x in {-1,...,1}{ \draw[violet, fill] (\x*5, 0) circle (2pt);}
\foreach \x in {-1,...,1}{ \draw[blue, fill] (\x*5+1, 0.5) circle (2pt);}
\foreach \x in {-1,...,1}{ \draw[cyan, fill] (\x*5+1, -0.5) circle (2pt);}
\foreach \x in {-1,...,1}{ \draw[red, fill] (\x*5+2, 0) circle (2pt);}
\foreach \x in {-1,...,1}{ \draw[teal, fill] (\x*5-1, 0) circle (2pt);}
\foreach \x in {-1,...,1}{ \draw[orange, fill] (\x*5-2, 0) circle (2pt);}

\draw (0,0)--node[below]{\tiny $a$} (0,0);
\draw (1,0.5)--node[above]{\tiny $ab$} (1,0.5);
\draw (1,-0.5)--node[below]{\tiny $ab'$} (1,-0.5);
\draw (2,0)--node[below]{\tiny $a\overline{b}c$} (2,0);
\draw (3,0)--node[below]{\tiny $a\overline{b}cd$} (3,0);
\draw (4,0)--node[below]{\tiny $a\overline{b}cde$} (4,0);
\draw (5,0)--node[below]{\tiny $a\overline{b}cdea$} (5,0);
\draw (6,0.5)--node[above]{\tiny $a\overline{b}cdeab$} (6,0.5);
\draw (6,-0.5)--node[below]{\tiny $a\overline{b}cdeab'$} (6,-0.5);
\draw (7,0)--node[below right]{\tiny $a\overline{b}cdea\overline{b}c$} (7,0);

\draw (-1,0) --node[below]{\tiny $ae$} (-1,0);
\draw (-2,0) --node[below]{\tiny $aed$} (-2,0);
\draw (-3,0) --node[below]{\tiny $aedc$} (-3,0);
\draw (-4,0.5) --node[above]{\tiny $aedcb$} (-4,0.5);
\draw (-4,-0.5) --node[below]{\tiny $aedcb'$} (-4,-0.5);
\draw (-5,0) --node[below]{\tiny $aedc\overline{b}a$} (-5,0);
\draw (-6,0) --node[above]{\tiny $aedc\overline{b}ae$} (-6,0);
\draw (-7,0) --node[below]{\tiny $aedc\overline{b}aed$} (-7,0);

\node at (0,-1.5){$U$};

\end{tikzpicture}};

\draw[->] (T) --node[below ]{$f$} (G);
\draw[->] (U) --node[above right]{$\rho$}(G);
\draw[->, dashed] (U) --node[above left]{$\exists\,\, !\,\, \tilde{\rho}$} (T);

\end{tikzpicture}
$$
\end{example}

\begin{proposition}[\cite{CS3}]\label{P:classifycover}  Any connected homotopy cover $\widetilde{G}$  is of the form  $U/S \cong \widetilde{G}$ where  $S \leq \Pi(G)$ is the subgroup $\{s | s\tilde{\rho} = \tilde{\rho}\}$. 
\end{proposition}

\begin{proposition}[\cite{CS3}]\label{P:index}
If $\widetilde{G} = U/S$ for a subgroup $S \leq \Pi(G)$ then the size of any preimage $f^{-1} (v)$ is the same as the index of  the subgroup $S \leq \Pi(G)$.  
\end{proposition}

\begin{example}
    Consider $G, \tilde{G}$ and $U$ from Examples \ref{Example:HCover}, \ref{Example:UCover}.   We see that $\tilde{G}\cong U/\langle (abcdea)^2\rangle$.
\end{example}

Lastly, in our computation of the fundamental group of Hamming graphs, we utilize the fundamental group of the complete graph.  This result and proof appears in an older version of \cite{CS3}, which may still be viewed via the arxiv.  We reproduce the result and proof below.

\begin{proposition}\cite{CS3}\label{P:Kn}
$\Pi(K_n)\cong \mathbb{Z}$ for $n=3$ and $ \mathbb{Z}/2$ for $n\geq 4$.
\end{proposition}

\begin{proof}
Any element of $\Pi(K_n)$ is represented by a walk $(vv_1v_2v_3\dots v)$.  Since $v_i \sim v$ for all $v_i \neq v$, this is equivalent to $(vv_1v_2vv_2v_3vv_3v_4v\dots) = (vv_1v_2v)*(vv_2v_3v)*(vv_3v_4v)*\dots)$.  Thus every element is a product of 3-cycles.  

Now choose a 3-cycle $\beta = (vv_1v_2v)$.  Since $n \geq 4$ there exists a $v_3$ which is not $v, v_1, v_2$ and then $\beta = (vv_1v_2v) = (vv_3v_2v) = (vv_3v_1v) = (vv_2v_1v) = \beta^{-1}$.  Given any other 3-cycle $\gamma = (vv_iv_jv)$, if $v_i \neq v_2$ then we have spider moves $\gamma = (vv_iv_jv) = (vv_iv_2v) = (vv_1v_2v)= \beta$.   If $v_i = v_2$ then we have $\gamma = (vv_2v_jv) =(vv_2v_1v) = \beta^{-1}$.  Thus every 3-cycle is equal to $\beta = \beta^{-1}$ and $\beta^2 = e$.   So $\Pi(K_n)\cong \mathbb{Z}/2$.
\end{proof}

\section{The Hamming graph $H(d,2)$.}\label{S:H(d,2)}

In this section, we explore the special case where there are two choices for entries of the tuples, or $q=2$, and classify the fundamental groups for these graphs.

\begin{lemma}
     If $d>2$ or $q>2$, then $H(d,q)$ is a pleat.
\end{lemma}
     
\begin{proof}
    We proceed by cases: In the first case, let $d>2$. Consider distinct vertices $(a_1,-,a_n)$ and $(x_1,-,x_n)$. If these vertices differ in all coordinates, observe that we can find a vertex $(x_1,-,a_j,-,a_n)$ such that $(a_1,-,a_n)\sim (x_1,-,a_j,-,a_n)$ and $(x_1,-,x_n) \not\sim (x_1,-,a_j,-,a_n)$. If these vertices share at least one coordinate, observe that for $k\neq j$, they can be written as $(a_1,-,a_k,-,a_j,-,a_n)$ and $(x_1,-,x_k,-,x_j,-,x_n)$, where $a_k=x_k$ but $a_j\neq x_j$. Then we can find a vertex $(a_1,-,b_k,-,a_j,-,a_n)$ such that $(a_1,-,a_k,-,a_j,-,a_n)\sim (a_1,-,b_k,-,a_j,-,a_n)$ and $(x_1,-,x_k,-,x_j,-,x_n)\not\sim (a_1,-,b_k,-,a_j,-,a_n)$. Thus, we have found a vertex in the neighborhood of $(a_1,-,a_n)$ that is not in the neighborhood of $(x_1,-,x_n)$, so $H(q,d)$ is a pleat when $d>2$. 

   In the second case, let $q>2$. Suppose we have vertices $(-,a,-,b,-)$ and $(-,c,-,d,-)$ where $b\neq d$ but $c$ could equal $a$. Since $q>2$, we can find a vertex $(-,x,-,d,-)$, where $x\neq a, c$.  Note that $(-,c,-,d,-)$ is in the neighborhood of $(-,x,-,d,-)$ but not in the neighborhood of $(-,a,-,b,-)$. So $H(d,q)$ is a pleat when $q>2$. 

   \end{proof}

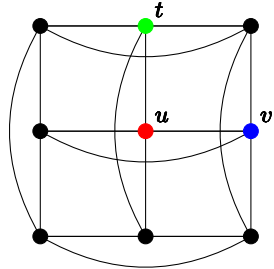
\begin{example} 
Observe the following graph of $H(2,2)$

\[\begin{tikzpicture}[scale=1.4]

\foreach \x in {0,...,1}{
\foreach \y in {0,...,1}{

\draw (\x,\y) --node[above right]{\tiny $(\x, \y)$} (\x, \y);

}
};
\foreach \x in {0,...,1}{
\draw (\x, 0)--(\x, 1);
};

\draw (0,0)--(1,0);
\draw (0,1)--(1,1);

\foreach \y in {0,...,1}{
\draw (0, \y)--(1, \y);
};

\draw[fill, black] (0,0) circle (2pt);
\draw[fill, black] (1,0) circle (2pt);
\draw[fill, black] (0,1) circle (2pt);
\draw[fill, black] (1,1) circle (2pt);11

 \end{tikzpicture}\]

Note that $(0,0)\sim(0,1)$ and $(0,0)\sim(1,0)$. Also note that $(1,1)\sim(0,1)$ and $(1,1)\sim(1,0)$. Since these two vertices share a neighborhood, $H(2,2)$ is not a pleat. 

\end{example}

   \begin{example}

\[\begin{tikzpicture}[scale=1.4]
\foreach \x in {0,...,2}{
\draw (\x, 0)--(\x, 2);
\draw (\x,0)edge[bend left](\x,2);
};

\draw (0,0)--(2,0);
\draw (0,2)--(2,2);
\draw (0,1)--(2,1);

\foreach \y in {0,...,2}{
\draw (0, \y)--(2, \y);
\draw (0,\y)edge[bend right](2,\y);
};

\draw[fill, black] (0,0) circle (2pt);
\draw[fill, black] (1,0) circle (2pt);
\draw[fill, black] (2,0) circle (2pt);
\draw[fill, black] (0,1) circle (2pt);
\draw[fill, black] (1,1) circle (2pt);
\draw[fill, black] (2,1) circle (2pt);
\draw[fill, black] (0,2) circle (2pt);
\draw[fill, black] (1,2) circle (2pt);
\draw[fill, black] (2,2) circle (2pt);

 \end{tikzpicture}\]
Observe in the above graph of $H(2,3)$ that for any distinct vertices $u, v$, since $u, v$ differ in at least one coordinate, there is always at least one vertex in the neighborhood of $u$ but not $v$.   Therefore, for all $u$ and $w$, $N(u) \not\subseteq N(w)$ so $H(2,3)$ is a pleat.

   \end{example}

We present a technical lemma describing the behavior of closed walks in $H(d,2)$.

\begin{lemma}\label{L:H2nDimensionReduction}
Let $W$ be a walk in $H(d,2)$ such that $W$ starts and ends at the origin.  Then $W$ is equivalent to a walk in $H(d-1,2)$ (as the subgraph of $H(d,2)$ induced by vertices whose last coordinate is $0$).
\end{lemma}

\begin{proof}
Let $W$ be an arbitrary walk of length $n$ starting and ending in $(0_1,-,0_n)$. We want to show that $W\simeq W'$ where $W'$ is a walk where every vertex in $W'$ has last coordinate 0. If $W=(\mathbf{0}v_1v_2\cdots v_k\cdots \mathbf{0})$ there is a lowest index $i$ such that $v_i=(x_1, -, x_{n-1}, 1)$, that is the $n$th coordinate is 1. If we isolate the part of the walk with $v_i$:

\begin{eqnarray*}
W&=&(\mathbf{0}\cdots v_{i-1}v_iv_{i+1}\cdots \mathbf{0})\\
&=&(\mathbf{0}\cdots (x_1, x_2,-, x_{n-1},0)(x_1, x_2,-, x_{n-1},1)v_{i+1}\cdots \mathbf{0}).
\end{eqnarray*}

The second equality follows from the fact that $i$ is the lowest index such that $v_i$ has a 1 in the $n$th coordinate.  Since $v_i, v_{i-1}$ differ by exactly one coordinate, $v_{i-1}$'s $n$th coordinate must be zero.

We proceed via induction on the number of vertices in $W$ with a 1 in the $n$th coordinate.  There are two cases for $v_{i+1}$. The first case is that $v_{i+1}$ has a 0 in the $n$th coordinate. This would make $v_{i+1}$ equal to  $v_{i-1}$, and  $v_i$ could be pruned. Then $W=(\mathbf{0}\cdots v_{i-1}v_iv_{i+1}\cdots \mathbf{0})\simeq (\mathbf{0}\cdots v_{i-1}v_{i+2}\cdots \mathbf{0})$. Thus the number of vertices in this walk with a 1 in the $n$th coordinate is reduced.

The second case is that $v_{i+1}$ has a 1 in the $n$th coordinate and is therefore different from $v_i$ in some other coordinate. Thus $v_{i+1}=(x_1,-,x_{j-1}, y_j, x_{j+1},-, 1)$, allowing 

\begin{eqnarray*}
W&=&(\mathbf{0}\cdots (x_1, -,x_j,-, x_{n-1},0)(x_1, -,x_j,-, x_{n-1},1)(x_1, -,y_j,-, x_{n-1},1)\cdots \mathbf{0})\\
&\simeq& (\mathbf{0}\cdots (x_1, -,x_j,-, x_{n-1},0)(x_1, -,y_j,-, x_{n-1},0)(x_1, -,y_j,-, x_{n-1},1)\cdots \mathbf{0}).
\end{eqnarray*}

We see that $(x_1, -,y_j,-, x_{n-1},0)$ differs from $v_{i-1}, v_{i+1}$ by one entry each, and thus this equivalence is a valid spider-move.  In either case, the number of vertices with a 1 in the $n$th coordinate has been reduced.

By induction, $W\simeq W'$ where $W'$ is a walk where every vertex in $W'$ has last coordinate 0.
\end{proof}

\begin{example}
    Consider the following example of $H(3,2)$.

     
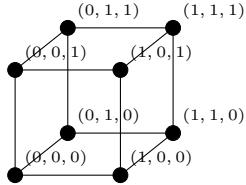
\[\begin{tikzpicture}[scale=1.4]

\foreach \x in {0,...,1}{
\foreach \y in {0,...,1}{

\draw (\x,\y) --node[above right]{\tiny $(\x, 0, \y)$} (\x, \y);

\draw[fill] (\x,\y) circle (2pt);

}
};
\foreach \x in {0,...,1}{
\draw (\x, 0)--(\x, 1);
};

\draw[ultra thick,green] (0, 0)--(1, 0);
\draw[ultra thick,blue] (0, 1)--(1, 1);

\def\sx{0.5}
\def\sy{0.4}

\foreach \x in {0,...,1}{
\foreach \y in {0,...,1}{

\draw (\x+\sx,\y+\sy) --node[above right]{\tiny $(\x, 1, \y)$} (\x+\sx, \y+\sy);

\draw[fill] (\x+\sx,\y+\sy) circle (2pt);
\draw[ultra thick, green] (\x, \y) -- (\x+\sx, \y+\sy);
}
\draw[ultra thick, blue] (\x, 1) -- (\x+\sx, 1+\sy);
};

\foreach \x in {0,...,1}{
\draw (\x+\sx, 0+\sy)--(\x+\sx, 1+\sy);
};


\draw[ultra thick, green] (0 + \sx,  0+\sy)--(1+\sx, 0+\sy);
\draw[ultra thick, blue] (0 + \sx,  1+\sy)--(1+\sx, 1+\sy);



 \end{tikzpicture}\]
 \end{example}

Observe the walk $W = (0,0,1)(0,1,1)(1,1,1)(1,0,1)$ pictured above in blue. Note that $(0,0,0),(0,1,1)\sim(0,0,1),(0,1,0)$, so $(0,1,1)$ can be moved to $(0,0,0)$ by a spider move. This also means $(0,0,1)$ can be moved to $(0,1,0)$ by a spider move. Observe that $(1,0,0),(1,1,1)\sim(1,0,1),(1,1,0)$, so $(1,1,1)$ can be moved to $(1,0,0)$ and $(1,0,1)$ can be moved to $(1,1,0)$ by spider moves as well. Then $W \simeq W'$, where $W' = (0,0,0)(0,1,0)(1,1,0)(1,0,0)$ depicted on the ``bottom" of the cube. 








 
The above lemma allows us to reduce the dimension that a closed walk `lives" in.  By applying this lemma with induction, we can show that all closed walks in $H(d,2)$ are trivial.

\begin{proposition}\label{L:H2nTrivial}
Let $W$ be a walk in $H(d,2)$ such that $W$ starts and ends at vertex $v_0=\mathbf{0}$. Then $W\simeq (\mathbf{0})$, which we call trivial. 
\end{proposition}

\begin{proof}
By repeatedly applying Lemma \ref{L:H2nDimensionReduction}, we show that $W\simeq W'$ which is a walk where only the first coordinate has non-zero entries, and such a walk must be of the form $\mathbf{0}(1,0,-,0)\mathbf{0}\cdots (1,0,-,0)\mathbf{0}$ which prunes completely.
\end{proof}

\begin{corollary}\label{C:trivial}
    $\Pi(H(d,2))\cong \{e\}$.
\end{corollary}
\begin{proof}
    Since each closed walk of $H(d,2)$ is trivial, its fundamental group has no nontrivial generators. 
\end{proof}

\section{General Hamming graphs $H(d,q)$, where $q\neq 2$.}\label{H(d,q)}

We begin our computation of $\Pi(H(d,q)), q>2$. We begin by defining the ground walk. The overall strategy will be to show that these ground walks form a minimal generating set of $\Pi(H(d,q))$. By also showing that they commute and after computing their order, we will be able to fully compute the fundamental group.

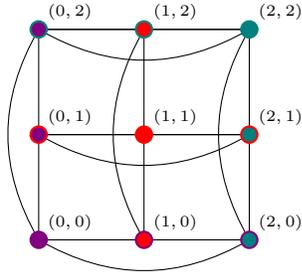
\[\begin{tikzpicture}[scale=1.4]

\foreach \x in {0,...,2}{
\foreach \y in {0,...,2}{

\draw (\x,\y) --node[above right]{\tiny $(\x, \y)$} (\x, \y);

}
};

\foreach \x in {0,...,2}{
\draw (\x, 0)--(\x, 2);
\draw (\x,0)edge[bend left](\x,2);
};

\draw (0,0)--(2,0);
\draw (0,2)--(2,2);
\draw (0,1)--(2,1);

\foreach \y in {0,...,2}{
\draw (0, \y)--(2, \y);
\draw (0,\y)edge[bend right](2,\y);
};

\draw[fill, violet] (0,0) circle (2pt);
\draw[fill, red] (1,0) circle (2pt);
\draw[fill, teal] (2,0) circle (2pt);
\draw[fill, violet] (0,1) circle (2pt);
\draw[fill, red] (1,1) circle (2pt);
\draw[fill, teal] (2,1) circle (2pt);
\draw[fill, violet] (0,2) circle (2pt);
\draw[fill, red] (1,2) circle (2pt);
\draw[fill, teal] (2,2) circle (2pt);

\draw[ thick, violet] (0,0) circle (2.2pt);
\draw[thick,  red] (0,1) circle (2.2pt);
\draw[ thick,  teal] (0,2) circle (2.2pt);
\draw[thick,  violet] (1,0) circle (2.2pt);
\draw[thick,  red] (1,1) circle (2.2pt);
\draw[thick,  teal] (1,2) circle (2.2pt);
\draw[thick,  violet] (2,0) circle (2.2pt);
\draw[thick,  red] (2,1) circle (2.2pt);
\draw[thick,  teal] (2,2) circle (2.2pt);
 \end{tikzpicture}\]

\begin{definition}\label{D:GroundWalk}
Let $U_i$ be a walk in $H(d,q)$ of length three where the only nonzero coordinate is $i$. Thus, $U_i=(\mathbf{0}(0,-,a_i,-,0)(0,-,b_i,-, 0)\mathbf{0})$. We call this a \textbf{ground walk} in the $i$th coordinate. 
\end{definition}

We will show that the ground walks defined above form a minimal generating set for the fundamental group. We begin by demonstrating that the ground walks commute as an important property of the group. We note the possibilities for the vertices of our ground walks: 
\begin{remark}\label{R:Uniqueness}
When $q=3$, the only options for $a_i$ and $b_i$ are 1 and 2, so there are only 2 options for $U_i$: $U_i=(\mathbf{0}(0,-, 1_i, -, 0)(0,-, 2_i, -, 0)\mathbf{0})$ or $U_i^{-1}=(\mathbf{0}(0,-, 2_i, -, 0)(0,-, 1_i, -, 0)\mathbf{0})$. We pick one, and one is the inverse of the other. 
If $q\neq 3$, there are multiple options for $a_i$ and $b_i$, so there are multiple possibilities for $U_i$ but note that any closed three-walks chosen are equivalent since these walks are contained in a subgraph isomorphic to $K_q$, and all closed three-walks are equivalent in $K_q$ from Proposition \ref{P:Kn} \cite{CS3}.
\end{remark}

\begin{example}\label{H(2,3)}

\[\begin{tikzpicture}[scale=1.4]

\foreach \x in {0,...,2}{
\foreach \y in {0,...,2}{

\draw (\x,\y) --node[above right]{\tiny $(\x, \y)$} (\x, \y);

}
};
\draw[ultra thick, blue] (0,0)--(1,0)--(2,0) edge[bend left] (0,0);
\draw[ultra thick, green] (0,0)--(0,1)--(0,2) edge[bend right] (0,0);

\foreach \x in {0,...,2}{
\draw (\x, 0)--(\x, 2);
\draw (\x,0)edge[bend left](\x,2);
};

\draw (0,0)--(2,0);
\draw (0,2)--(2,2);
\draw (0,1)--(2,1);
\foreach \y in {0,...,2}{
\draw (0, \y)--(2, \y);
\draw (0,\y)edge[bend right](2,\y);
};

\draw[fill, violet] (0,0) circle (2pt);
\draw[fill, red] (1,0) circle (2pt);
\draw[fill, teal] (2,0) circle (2pt);
\draw[fill, violet] (0,1) circle (2pt);
\draw[fill, red] (1,1) circle (2pt);
\draw[fill, teal] (2,1) circle (2pt);
\draw[fill, violet] (0,2) circle (2pt);
\draw[fill, red] (1,2) circle (2pt);
\draw[fill, teal] (2,2) circle (2pt);

\draw[ thick, violet] (0,0) circle (2.2pt);
\draw[thick,  red] (0,1) circle (2.2pt);
\draw[ thick,  teal] (0,2) circle (2.2pt);
\draw[thick,  violet] (1,0) circle (2.2pt);
\draw[thick,  red] (1,1) circle (2.2pt);
\draw[thick,  teal] (1,2) circle (2.2pt);
\draw[thick,  violet] (2,0) circle (2.2pt);
\draw[thick,  red] (2,1) circle (2.2pt);
\draw[thick,  teal] (2,2) circle (2.2pt);
 \end{tikzpicture}\]
 \end{example}
Observe the following example of $H(2,3)$. Recall that the ground walks $(0,0),(0,1),(0,2),(0,0)$ and $(0,0),(1,0),(2,0),(0,0)$ are walks of length 3. Then, any equivalent walk will be of odd length. Since pruning preserves the parity of length, such a walk can never be equivalent to a walk of length 0. Thus, the ground walks are not trivial.

\begin{lemma}\label{L:Commutativity}
Let $U_i$ and $U_j$ be ground walks in $H(d,q)$ such that $U_i=(\mathbf{0} (0,-, a_i, -, 0)(0,-, b_i, -, 0)\mathbf{0})$ and $U_j=(\mathbf{0}(0,-, a_j, -, 0)(0,-, b_j, -, 0)\mathbf{0})$, $i\neq j$. Then $U_i*U_j\simeq U_j*U_i$, i.e. $U_i, U_j$ commute.
\end{lemma}

\begin{proof}
Let $W=U_i*U_j$, then:
\begin{eqnarray*}
W&=&\mathbf{0}(0,-, a_i, -, 0)(0,-, b_i, -,0)\mathbf{0}(0,-, a_j, -, 0)(0,-, b_j, - )\mathbf{0}\\
&\simeq&\mathbf{0}(0,-, a_i, -, 0)(0,-, b_i, -, 0)(0,-,b_i,-, a_j, -, 0), (0,-, a_j, -, 0)(0,-, b_j,-,0)\mathbf{0}\\
&\simeq&\mathbf{0}(0,-, a_i, -, 0)(0,-, a_i, -, a_j, -,0 )(0,-,b_i,-, a_j, -, 0)(0,-, b_i, -, b_j, -, 0)(0,-, b_j, -,0 )\mathbf{0}\\
&\simeq&\mathbf{0}(0,-, a_j, -, 0)(0,-, a_i,-, a_j, -,0 )(0,-,b_i,-, a_j, -, 0)(0,-, b_i, -, b_j, -, 0)(0,-, b_i, -,0 )\mathbf{0}\\
&\simeq&\mathbf{0}(0,-, a_j, -, 0)(0,-, a_i, -, a_j, -,0 )(0,-,a_i,-, b_j, -, 0)(0,-, b_i, -, b_j, -, 0)(0,-, b_i, -,0 )\mathbf{0}\\
&\simeq&\mathbf{0}(0,-, a_j, -, 0)(0,-, b_j, -, 0)(0,-,a_i,-, b_j, -, 0)(0,-, a_i, -, 0)(0,-, b_i, -, 0)\mathbf{0}\\
&\simeq&\mathbf{0}(0,-, a_j, -, 0)(0,-, b_j, -, 0)\mathbf{0}(0,-, a_i, -, 0)(0,-, b_i, -, 0)\mathbf{0}\\
&=&U_j*U_i.
\end{eqnarray*}
Through spider moves, we are able to reorder the vertices so we move through the ground walks in the reverse order, therefore showing that $U_i$ and $U_j$ commute.
\end{proof}

We next present a technical lemma showing what spider moves between walks are possible, which will later be used to show that as a generating set, the ground walks are minimal.

\begin{lemma}\label{L:AllowedSpider}
Let $W, W'$ be walks which differ by a spider move on the $k$th vertex of each walk, $v_k, v'_k$ respectively. If $v_{k-1}, v_k$ and $v_k, v_{k+1}$ differ in the same coordinate, then $v'_k$ is different from both vertices in the same coordinate as well.  However if $v_{k-1}, v_k$ differ in coordinate $i$ and $v_k, v_{k+1}$ differ in coordinate $j$, then $v_{k-1}, v'_k$ differ in coordinate $j$ and $v'_k, v_{k+1}$ differ in coordinate $i$, and there is only one choice for $v'_k$.
\end{lemma}

\begin{proof}
    Suppose in the first case $v_{k-1} = (a_1, - a_i, -, a_n ), v_k = (a_1, - b_i, -, a_n )$, $v_{k+1} = (a_1, - c_i, -, a_n )$. Then since $v'_k\sim v_{k-1}, v_{k+1}$ it must be of the form $(a_1, - d_i, -, a_n )$.  Otherwise, either $v_{k-1}, v'_k$ or $v'_k, v_{k+1}$ would differ in more than 2 coordinates.
    
    Suppose in the second case: $v_{k-1} = (a_1, - , a_i, - a_j, -, a_n ), v_k = (a_1, - ,b_i, -, a_j, -, a_n )$, $v_{k+1} = (a_1, - b_i, -, b_j, -, a_n )$.  Then since $v'_k\sim v_{k-1}, v_{k+1}$ and $v'_k\neq v_k$, the only possible choice for $v'_k$ is $v'_k=(a_1, - a_i, -, b_j, -, a_n)$
\end{proof}

\begin{example}

    Case 1: Three consecutive vertices differ in the same coordinate

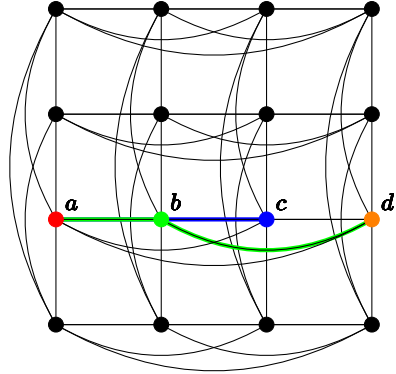
\[\begin{tikzpicture}[scale=1.4]

\foreach \x in {0,...,3}{
    \foreach \y in {0,...,3}{

    \draw (0,1) --node[above right]{\small $a$} (0, 1);
    \draw (1,1) --node[above right]{\small $b$} (1, 1);
    \draw (2,1) --node[above right]{\small $c$} (2, 1);
    \draw (3,1) --node[above right]{\small $d$} (3, 1);

    };
};

\foreach \x in {0,...,3}{
\draw (\x, 0)--(\x, 3);
\draw (\x,0)edge[bend left](\x,2);
\draw (\x,1)edge[bend left](\x,3);
\draw (\x,0)edge[bend left](\x,3);

};

\draw (0,0)--(3,0);
\draw (0,2)--(3,2);
\draw (0,1)--(3,1);
\draw (0,3)--(3,3);

\draw[ultra thick, blue] (0,1)--(1,1)--(2,1);
\draw[ultra thick, green] (0,1)--(1,1) edge[bend right] (3,1);

\foreach \y in {0,...,3}{
\draw (0, \y)--(3, \y);
\draw (0,\y)edge[bend right](2,\y);
\draw (1,\y)edge[bend right](3,\y);
\draw (0,\y)edge[bend right](3,\y);

};

\draw[fill, black] (0,0) circle (2pt);
\draw[fill, black] (1,0) circle (2pt);
\draw[fill, black] (2,0) circle (2pt);
\draw[fill, red] (0,1) circle (2pt);
\draw[fill, green] (1,1) circle (2pt);
\draw[fill, blue] (2,1) circle (2pt);
\draw[fill, black] (0,2) circle (2pt);
\draw[fill, black] (1,2) circle (2pt);
\draw[fill, black] (2,2) circle (2pt);
\draw[fill, black] (3,0) circle (2pt);
\draw[fill, orange] (3,1) circle (2pt);
\draw[fill, black] (3,2) circle (2pt);
\draw[fill, black] (0,3) circle (2pt);
\draw[fill, black] (1,3) circle (2pt);
\draw[fill, black] (2,3) circle (2pt);
\draw[fill, black] (3,3) circle (2pt);

 \end{tikzpicture}\]

 Observe the following example of $H(2,4)$, where we have walk $W=(a,b,c)$. Note that $a,b,c$ are three consecutive vertices that differ from one another in the same coordinate. Observe then that to create walk $W'\simeq W$ where $W$ and $W'$ differ by a spider move, $W'=(a,b,d)$ is the only option, where we have a spider move changing vertex $c$ to $d$. 
    
    Case 2: Three consecutive vertices differ in multiple coordinates
     
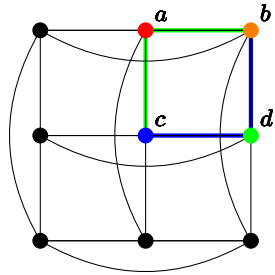
\[\begin{tikzpicture}[scale=1.4]

\foreach \x in {0,...,2}{
\foreach \y in {0,...,2}{

\draw (1,2) --node[above right]{\small $a$} (1, 2);
\draw (2,2) --node[above right]{\small $b$} (2, 2);
\draw (1,1) --node[above right]{\small $c$} (1, 1);
\draw (2,1) --node[above right]{\small $d$} (2, 1);

}
};

\draw[ultra thick, blue] (1,1)--(2,1)--(2,2);
\draw[ultra thick, green] (1,1)--(1,2)--(2,2);

\foreach \x in {0,...,2}{
\draw (\x, 0)--(\x, 2);
\draw (\x,0)edge[bend left](\x,2);
};

\draw (0,0)--(2,0);
\draw (0,2)--(2,2);
\draw (0,1)--(2,1);
\foreach \y in {0,...,2}{
\draw (0, \y)--(2, \y);
\draw (0,\y)edge[bend right](2,\y);
};

\draw[fill, black] (0,0) circle (2pt);
\draw[fill, black] (1,0) circle (2pt);
\draw[fill, black] (2,0) circle (2pt);
\draw[fill, black] (0,1) circle (2pt);
\draw[fill, blue] (1,1) circle (2pt);
\draw[fill, green] (2,1) circle (2pt);
\draw[fill, black] (0,2) circle (2pt);
\draw[fill, red] (1,2) circle (2pt);
\draw[fill, orange] (2,2) circle (2pt);
 \end{tikzpicture}\]
 
Observe in the following example of $H(2,3)$ the path $W=(c,a,b)$. We see that $a\sim c,b$, so the only possible spider move for $a$ is $d$, since $d\sim c,b$ as well. So $W'=(c,d,b)$, which differs from $W$ by a spider move. 
\end{example}

 We now show that the ground walks are minimal as a generating set.

\begin{lemma}\label{L:GroundWalksDifferent}
For a walk $U_i$, $U_i$ can not be a product of all the $U_j$ where $i\neq j$.
\end{lemma}

\begin{proof}
    Note that $U_i = \mathbf{0}, (0,-, a_i, -, 0)(0,-, b_i, - 0),\mathbf{0}$ and each $U_j$ is defined similarly.  Note that $U_i$ is a walk where the $i$th coordinate is changed 3 times, and each other coordinate changes 0 times.  Note that for any walk $W$ and $W'$, a prune or anti-prune of $W$, $W'$ has either 2 less or 2 more changes in some coordinate $k$ than $W$.  
    
    Suppose $W'$ and $W$ differ by a spider move.  Then by Lemma \ref{L:AllowedSpider}, the number of changes in each coordinate between $W, W'$ remain exactly the same.
    
    Then for any $\hat{W}\simeq W$, since $\hat{W}$ is achieved by a finite sequence of prunes, anti-prunes or spider moves, the number of changes in each coordinate has the same parity as $W$.
    
    So if $\hat{W}\simeq U_i$ then $\hat{W}$ has an odd number of changes in coordinate $i$ and an even number otherwise.  But since all the $U_j$ have 0 changes in coordinate $i$, we have proven that $U_i$ is not generated by the set $\{U_j|j \neq i\}$.
\end{proof}

We  will now show that the subgroup they generate is in fact the entire fundamental group.  The following proposition shows that we may ``retract" a 3-cycle mid walk back to a ground walk.  

\begin{proposition}\label{L:ThreeWalkRetract}
Let $W$ be a walk of length n in $H(d,q)$ starting and ending at the origin such that $$W = K*((a_1,-,a_i,-,a_d)(-,b_i,-)(-,c_i,-)(-,a_i,-))*K^{-1}.$$ Let $C=((-,a_i,-)(-,b_i,-)(-,c_i,-)(-,a_i,-))$, and thus $W=K*C*K^{-1}$. Then $C$ is equivalent to a ground walk $U_i$.
\end{proposition}

\begin{proof}
We proceed via induction on the length of $W$. \\
For a walk $W$ where $n=3$, $W=(\mathbf{0}(0,- a_i-,0)(0-,b_i,-,0)(0-,0))$. Since the walk must start and end at the origin, the only option for $W$ is $U_i$.\\

There are two cases for the vertex leading up to $C$ within $W$ for walks longer than $n=3$. The first case is that the preceding vertex also has a change in the same coordinate as $C$. In this case:

\begin{eqnarray*}
W&=&(\mathbf{0}\cdots (-,x_i,-)(-,a_i,-)(-,b_i,-)(-,c_i,-)(-,a_i,-)(-,x_i,-)\cdots\mathbf{0})\\
&\simeq&(\mathbf{0}\cdots (-,x_i,-)(-,a_i,-)(-,x_i,-)(-,c_i,-)(-,a_i,-)(-,x_i,-)\cdots\mathbf{0})\\
&\simeq&(\mathbf{0}\cdots (-,x_i,-)(-,c_i,-)(-,a_i,-)(-,x_i,-)\cdots\mathbf{0})\\
&\simeq&(\mathbf{0}\cdots x_i)*(-,x_i,-)(-,c_i,-)(-,a_i,-)(-,x_i,-)*(x_i\cdots\mathbf{0}).\\
\end{eqnarray*}

The second case is that the vertex before does not have a change in the same coordinate. The inductive hypothesis is the same, but in this case we have:

\begin{eqnarray*}
W&=&(\mathbf{0}\cdots(-,a_i,-,a_j,-)(-,a_i,-,b_j,-)(-,b_i,-,b_j,-)(-,c_i,-,b_j,-)(-,a_i,-,a_j-)(-,a_i,-,x_j,-) \cdots\mathbf{0})\\
&\simeq&(\mathbf{0}\cdots(-,a_i,-,a_j,-)(-,b_i,-,a_j,-)(-,b_i,-,b_j,-)(-,c_i,-,b_j,-)(-,a_i,-,b_j-)(-,a_i,-,a_j,-) \cdots\mathbf{0})\\
&\simeq&(\mathbf{0}\cdots(-,a_i,-,a_j,-)(-,b_i,-,a_j,-)(-,c_i,-,a_j,-)(-,c_i,-,b_j,-)(-,a_i,-,b_j,-)(-,a_i,-,a_j,-) \cdots\mathbf{0})\\
&\simeq&(\mathbf{0}\cdots(-,a_i,-,a_j,-)(-,b_i,-,a_j,-)(-,c_i,-,a_j-)(-,a_i,-,a_j,-)(-,a_i-,b_j,-)(-,a_i,-,a_j,-) \cdots\mathbf{0})\\
&\simeq&(\mathbf{0}\cdots(-,a_i,-,a_j,-)(-,b_i,-,a_j,-)(-,c_i,-,a_j,-)(-,a_i,-,a_j,-)\cdots\mathbf{0}).\\
\end{eqnarray*}

In both cases, the final equation gives us a walk shorter than $W$ that contains a $C$, which we know is equivalent to $U_i$ by induction. \\
\end{proof}


We present another technical lemma which allows us to rearrange when we have changes in coordinates between vertices in our walks.  
\begin{lemma}\label{L:ReorderConsecutiveChanges}
Let $W$ be a walk of length $n$ such that there is a change in the $i$th coordinate followed by a change in the $j$th coordinate, $j\neq i$.  Then $W\simeq W'$ such that $W$ and $W'$ are the same, except that the change in $i$th coordinate follows the change in $j$th coordinate.
\begin{proof}
So for a walk $W$ consider:
\begin{eqnarray*}
W&=&(\mathbf{0}\cdots(a_1,-,a_i,-,a_j,-,a_n)(a_1,-,b_i,-,a_j,-,a_n)(a_1,-,b_i,-,c_j,-,a_n)\cdots\mathbf{0}) \\
&\simeq&(\mathbf{0}\cdots(a_1,-,a_i,-,a_j,-,a_n)(a_1,-,a_i,-,c_j,-,a_n)(a_1,-,b_i,-,c_j,-,a_n)\cdots\mathbf{0})\\
&=&W'.
\end{eqnarray*}
The second equivalence follows from a spider move. 

Note that $W'$ is identical to $W$ except for these 2 vertices, commuting the changes in coordinates $i$ and $j$.
\end{proof}

\end{lemma}

A consequence of Lemma \ref{L:ReorderConsecutiveChanges} is that any pair of consecutive changes in different coordinates may be swapped while preserving equivalence.\\

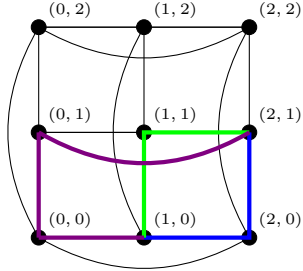
\begin{example}\label{E:changecommute}
    \[\begin{tikzpicture}[scale=1.4]

\foreach \x in {0,...,2}{
\foreach \y in {0,...,2}{

\draw (\x,\y) --node[above right]{\tiny $(\x, \y)$} (\x, \y);

}
};

\foreach \x in {0,...,2}{
\draw (\x, 0)--(\x, 2);
\draw (\x,0)edge[bend left](\x,2);
};

\draw (0,0)--(2,0);
\draw (0,2)--(2,2);
\draw (0,1)--(2,1);
\foreach \y in {0,...,2}{
\draw (0, \y)--(2, \y);
\draw (0,\y)edge[bend right](2,\y);
};

\draw[fill, black] (0,0) circle (2pt);
\draw[fill, black] (1,0) circle (2pt);
\draw[fill, black] (2,0) circle (2pt);
\draw[fill, black] (0,1) circle (2pt);
\draw[fill, black] (1,1) circle (2pt);
\draw[fill, black] (2,1) circle (2pt);
\draw[fill, black] (0,2) circle (2pt);
\draw[fill, black] (1,2) circle (2pt);
\draw[fill, black] (2,2) circle (2pt);

\draw[ultra thick, blue] (1,0)--(2,0)--(2,1); 
\draw[ultra thick, green] (1,0)--(1,1)--(2,1); 
\draw[ultra thick, violet] (0,0)--(1,0); 
\draw[ultra thick, violet] (2,1) edge [bend left] (0,1);
\draw[ultra thick, violet] (0,1)--(0,0);

\end{tikzpicture}\]
 Observe the walks $W$ and $V$ in the above example of $H(2,3)$, where $W=(0,0)(1,0)(2,0)(2,1)(0,1)(0,0)$ and $V=(0,0)(1,0)(1,1)(2,1)(0,1)(0,0)$. Note that $W\simeq V$ since the walks only differ by a spider move. Isolating this spider move, in $W$ we have $(1,0)(2,0)(2,1)$ and in $V$ we have $(1,0)(1,1)(2,1)$ Note that in $W$, there is a change in the first coordinate $i$ followed by a change in the second coordinate $j$. The opposite is true for $V$, where we have a change in $j$ followed by a change in $i$. Since we still start at $(1,0)$ and end at $(2,1)$, we can commute the changes in coordinates $i$ and $j$.
\end{example}

Using the allowed spider moves and reordering properties of Hamming graphs, we are now able to show that the ground walks generate $\Pi(H(d,q))$.

\begin{theorem}\label{T:Generate}
    Let $W$ be a walk in $H(d,q)$ such that $W$ starts and ends at the origin. Then $W$ is equivalent to a walk $W'$ that is the product of ground walks. (This means that $W$ is generated by the ground walks.) 
\end{theorem}
\begin{proof}
    We proceed by induction on the length of $W$.
For $n=3$, let $W=(\mathbf{0}(0,-, a_i-,0)(0,-,b_i,-,0)\mathbf{0})$. The only possibility for a walk of length 3 that will start and end at the origin is $U_i$, or it's inverse, which is a ground walk.\\

For $n>3$, let $W$ be an arbitrary walk of length n that starts and ends at the origin. We want to show $W$ is a product of walks generated by the ground walks. \\
Note that there are two cases. The first case is that somewhere in $W$, there exists two consecutive changes in the same coordinate. In that case:

\begin{eqnarray*}
W&=&(\mathbf{0}\cdots(a_1, -, a_i, -, a_d)(a_1, -, b_i, -, a_d)(a_1, -, c_i, -, a_d) \cdots \mathbf{0} )\\
&\simeq& C *(a_1, -, a_i, -, a_d)(a_1, -, b_i, -, a_d)(a_1, -, c_i, -, a_d)* K \\
&\simeq& (C*(a_1, -, a_i, -, a_d)(a_1, -, b_i, -, a_d)(a_1, -, c_i, -, a_d)(a_1, -, a_i, -, a_d)* C^{-1})\\ 
&&*(C* (a_1, -, a_i, -, a_d)(a_1, -, c_i, -, a_d)* K).\\
\end{eqnarray*}

In this last equivalence, $(C*(a_1, -, a_i, -, a_d)(a_1, -, b_i, -, a_d)(a_1, -, c_i, -, a_d)(a_1, -, a_i, -, a_d)* C^{-1})$ is equivalent to a ground walk by Proposition \ref{L:ThreeWalkRetract}. We are therefore left with a walk $(C* (a_1, -, a_i, -, a_d)(a_1, -, c_i, -, a_d)* K)$ shorter than $W$ that starts and ends at the origin, which by induction is generated by the ground walks. 
\\ 
    
The second case is that there are not two consecutive changes in the same coordinate in $W$. Since our walk must start and end at the origin, if we've made a change to  coordinate $i$ in our walk, there must be a subsequent change in $i$ in our walk. In that case:
By Lemma \ref{L:ReorderConsecutiveChanges}, we can permute the changes in the walk such that the changes in the $i$th coordinate are consecutive. We now have the case where there are two consecutive changes in the same coordinate, so this walk is also generated by the ground walks by induction.
\end{proof}

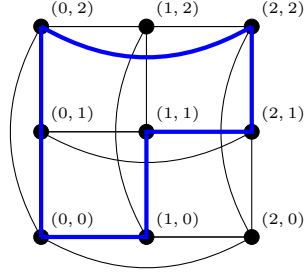
\begin{example}\label{generate}
    \[\begin{tikzpicture}[scale=1.4]

\foreach \x in {0,...,2}{
\foreach \y in {0,...,2}{

\draw (\x,\y) --node[above right]{\tiny $(\x, \y)$} (\x, \y);

}
};

\foreach \x in {0,...,2}{
\draw (\x, 0)--(\x, 2);
\draw (\x,0)edge[bend left](\x,2);
};

\draw (0,0)--(2,0);
\draw (0,2)--(2,2);
\draw (0,1)--(2,1);
\foreach \y in {0,...,2}{
\draw (0, \y)--(2, \y);
\draw (0,\y)edge[bend right](2,\y);
};

\draw[fill, black] (0,0) circle (2pt);
\draw[fill, black] (1,0) circle (2pt);
\draw[fill, black] (2,0) circle (2pt);
\draw[fill, black] (0,1) circle (2pt);
\draw[fill, black] (1,1) circle (2pt);
\draw[fill, black] (2,1) circle (2pt);
\draw[fill, black] (0,2) circle (2pt);
\draw[fill, black] (1,2) circle (2pt);
\draw[fill, black] (2,2) circle (2pt);

\draw[ultra thick, blue] (0,0)--(1,0)--(1,1)--(2,1)--(2,2) edge[bend left] (0,2);
\draw[ultra thick, blue] (0,2)--(0,1)--(0,0);

\end{tikzpicture}\]
Pictured above is the walk $W=(0,0)(1,0)(1,1)(2,1)(2,2)(0,2)(0,1)(0,0)$.
Observe the following series of spider moves and prunes:
\begin{eqnarray*}
W&=&(0,0)(1,0)(1,1)(2,1)(2,2)(0,2)(0,1)(0,0)\\
&\simeq&(0,0)(1,0)(2,0)(2,1)(2,2)(0,2)(0,1)(0,0)\\
&\simeq&(0,0)(1,0)(2,0)(2,1)(0,1)(0,2)(0,1)(0,0)\\
&\simeq&(0,0)(1,0)(2,0)(0,0)(0,1)(0,2)(0,1)(0,0)\\
&\simeq&(0,0)(1,0)(2,0)(0,0)(0,1)(0,2)(0,1)(0,0)\\
&\simeq&(0,0)(1,0)(2,0)(0,0)(0,1)(0,0)\\
&\simeq&(0,0)(1,0)(2,0)(0,0).
\end{eqnarray*}
This equivalence produces the following graph:
    
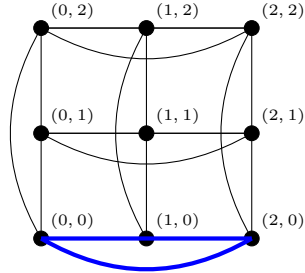
\[\begin{tikzpicture}[scale=1.4]

\foreach \x in {0,...,2}{
\foreach \y in {0,...,2}{

\draw (\x,\y) --node[above right]{\tiny $(\x, \y)$} (\x, \y);

}
};

\foreach \x in {0,...,2}{
\draw (\x, 0)--(\x, 2);
\draw (\x,0)edge[bend left](\x,2);
};

\draw (0,0)--(2,0);
\draw (0,2)--(2,2);
\draw (0,1)--(2,1);
\foreach \y in {0,...,2}{
\draw (0, \y)--(2, \y);
\draw (0,\y)edge[bend right](2,\y);
};

\draw[fill, black] (0,0) circle (2pt);
\draw[fill, black] (1,0) circle (2pt);
\draw[fill, black] (2,0) circle (2pt);
\draw[fill, black] (0,1) circle (2pt);
\draw[fill, black] (1,1) circle (2pt);
\draw[fill, black] (2,1) circle (2pt);
\draw[fill, black] (0,2) circle (2pt);
\draw[fill, black] (1,2) circle (2pt);
\draw[fill, black] (2,2) circle (2pt);

\draw[ultra thick, blue] (0,0)--(1,0)--(2,0) edge[bend left] (0,0);

\end{tikzpicture}\]

Note that this is a graph consisting of one of the ground walks, thus confirming that $W$ is equivalent to a product of ground walks. 

By Lemma \ref{L:GroundWalksDifferent}, the ground walks are not just a generating set for $\Pi(H(d,q))$, but a minimal generating set. 
 We now address the order of the ground walks.  Note that the ground walk $U_i$ is fully contained in a subgraph where only the $i$th coordinate is allowed to be non-zero, which is isomorphic to $K_q$, and we will use this fact in our computations.

\end{example}
\begin{lemma}\label{L:Order}
Let $U_i$ be a ground walk in $\Pi(H(d,q))$. Then, the order of $U_i$ is infinite when $q=3$ and the order of $U_i$ is 2 when $q\neq3$. 
\end{lemma}
\begin{proof}
We begin with the case where $q=3$. In this case, consider the following cover of $H(3,2)$:

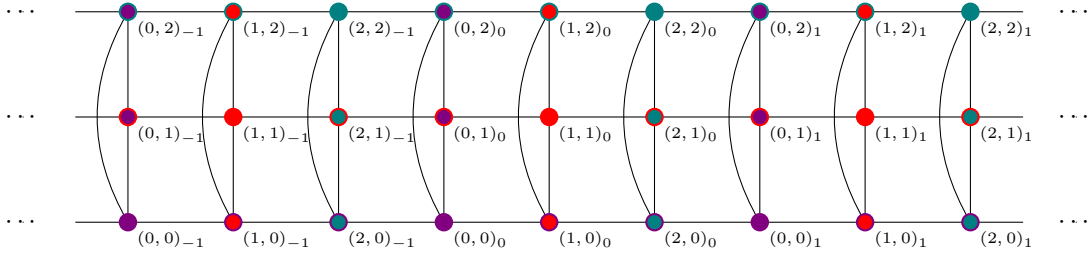
\[\begin{tikzpicture}[scale=1.4]

\foreach \i in {-1,...,1}{
\foreach \x in {0,...,2}{
\foreach \y in {0,...,2}{

\draw (\x+3*\i,\y) --node[below right]{\tiny $(\x, \y)_{\i}$} (\x+3*\i, \y);

}
}
};

\foreach \y in {0,...,2}{
\draw (-4,\y) --node{$\cdots$} (-4, \y);
\draw (6,\y) --node{$\cdots$} (6, \y);
};

\foreach \x in {-3,...,5}{
\draw (\x, 0)--(\x, 2);
\draw (\x,0)edge[bend left](\x,2);
};

\draw (-3.5,0)--(5.5,0);
\draw (-3.5,2)--(5.5,2);
\draw (-3.5,1)--(5.5,1);

\draw[fill, violet] (0-3,0) circle (2pt);
\draw[fill, red] (1-3,0) circle (2pt);
\draw[fill, teal] (2-3,0) circle (2pt);
\draw[fill, violet] (0-3,1) circle (2pt);
\draw[fill, red] (1-3,1) circle (2pt);
\draw[fill, teal] (2-3,1) circle (2pt);
\draw[fill, violet] (0-3,2) circle (2pt);
\draw[fill, red] (1-3,2) circle (2pt);
\draw[fill, teal] (2-3,2) circle (2pt);

\draw[ thick, violet] (0-3,0) circle (2.2pt);
\draw[thick,  red] (0-3,1) circle (2.2pt);
\draw[ thick,  teal] (0-3,2) circle (2.2pt);
\draw[thick,  violet] (1-3,0) circle (2.2pt);
\draw[thick,  red] (1-3,1) circle (2.2pt);
\draw[thick,  teal] (1-3,2) circle (2.2pt);
\draw[thick,  violet] (2-3,0) circle (2.2pt);
\draw[thick,  red] (2-3,1) circle (2.2pt);
\draw[thick,  teal] (2-3,2) circle (2.2pt);


\draw[fill, violet] (0,0) circle (2pt);
\draw[fill, red] (1,0) circle (2pt);
\draw[fill, teal] (2,0) circle (2pt);
\draw[fill, violet] (0,1) circle (2pt);
\draw[fill, red] (1,1) circle (2pt);
\draw[fill, teal] (2,1) circle (2pt);
\draw[fill, violet] (0,2) circle (2pt);
\draw[fill, red] (1,2) circle (2pt);
\draw[fill, teal] (2,2) circle (2pt);

\draw[ thick, violet] (0,0) circle (2.2pt);
\draw[thick,  red] (0,1) circle (2.2pt);
\draw[ thick,  teal] (0,2) circle (2.2pt);
\draw[thick,  violet] (1,0) circle (2.2pt);
\draw[thick,  red] (1,1) circle (2.2pt);
\draw[thick,  teal] (1,2) circle (2.2pt);
\draw[thick,  violet] (2,0) circle (2.2pt);
\draw[thick,  red] (2,1) circle (2.2pt);
\draw[thick,  teal] (2,2) circle (2.2pt);


\draw[fill, violet] (0+3,0) circle (2pt);
\draw[fill, red] (1+3,0) circle (2pt);
\draw[fill, teal] (2+3,0) circle (2pt);
\draw[fill, violet] (0+3,1) circle (2pt);
\draw[fill, red] (1+3,1) circle (2pt);
\draw[fill, teal] (2+3,1) circle (2pt);
\draw[fill, violet] (0+3,2) circle (2pt);
\draw[fill, red] (1+3,2) circle (2pt);
\draw[fill, teal] (2+3,2) circle (2pt);

\draw[ thick, violet] (0+3,0) circle (2.2pt);
\draw[thick,  red] (0+3,1) circle (2.2pt);
\draw[ thick,  teal] (0+3,2) circle (2.2pt);
\draw[thick,  violet] (1+3,0) circle (2.2pt);
\draw[thick,  red] (1+3,1) circle (2.2pt);
\draw[thick,  teal] (1+3,2) circle (2.2pt);
\draw[thick,  violet] (2+3,0) circle (2.2pt);
\draw[thick,  red] (2+3,1) circle (2.2pt);
\draw[thick,  teal] (2+3,2) circle (2.2pt);

\end{tikzpicture}\]

To show that this is a cover, we need a covering map so that $N((a,b)_i)\in V(C)$ maps bijectively to $N((a,b))$, and each 4-cycle in $H(3,2)$ lifts to a  4-cycle in $C$. We note that this graph $C$ has vertex set $\{v_i: v\in V(H(2,3)), i\in \mathbb{Z}\}$ and $v_i\sim w_j$ if $v=(a,x), w=(a,y)$ and $i=j$, if $v=(1,x), w=(a,x), a\neq 1$ and $i=j$ or $v=(2,x), w=(0,x), j=i+1$.  Since each 4-cycle in $H(2,3)$ has vertices of the form $(a,x), (b,x), (b,y), (a,y)$ one can verify that a closed 4-walk $(a,x)(b,x)(b,y)(a,y)$ lifts to a closed 4-walk in $C$. Thus,  $C$ is a cover with covering map $\rho:C\to H(2,3), v_i\mapsto v$.  Note that the ground walk $((0,0)(1,0)(2,0)(0,0))$ lifts to a non-closed walk $((0,0)_i(1,0)_i,(2,0)_i(0,0)_{i+1})$ but $((0,0)(0,1)(0,2)(0,0))$ lifts to closed walk $((0,0)_i(0,1)_i,(0,2)_i(0,0)_{i})$.  We also note that $C$ is infinite, so by Proposition \ref{P:index}, we have that the subgroup of $\Pi(H(2,3))$ generated by $U_2=((0,0)(0,1)(0,2)(0,0))$ has infinite index in $\Pi(H(2,3))$ so the subgroup generated by $U_1=((0,0)(1,0)(2,0)(0,0))$ has infinite order and thus $U_1$ has infinite order.  A similar argument shows that $U_2$ has infinite order as well.  Generalizing this to higher dimension produces tedious but similar arguments.



In the second case, where $q\neq 3$, note that the induced subgraph of vertices $\{(0,-, x_i,-0): 0\leq x_i\leq q-1\}$ is isomorphic to $K_q$. Then , by Proposition \ref{P:Kn} any length 3 closed walk within this subgraph has order 2 since $q>3$.  Since $U_i$ is a closed 3-walk in this subgraph, $U_i$,  has order 2. 
\end{proof}

Using the properties we have demonstrated in this section, we can now compute the fundamental groups for all Hamming graphs.

\begin{theorem}\label{T:FundamentalGroups}
    $\Pi(H(d,3))\cong \mathbb{Z}^d$, and  $\Pi(H(d,q))\cong \mathbb{Z}_2^d$ when $q\neq 3$.
\end{theorem}
\begin{proof}
    For a walk $W$ in $H(d,q)$, we know that $W$ is generated by the ground walks by Theorem \ref{T:Generate}. We also know that the number of generators is at most $d$, since $d$ is the number of ground walks. Lemma \ref{L:GroundWalksDifferent} tells use that $U_i$ is not generated by the other $U_j$ $j\neq i$, so the number of generators is at least $d$. Thus, the minimal number of generators is $d$. We know that the walks commute from Lemma \ref{L:Commutativity}, and thus $\Pi(H(d,q))$ is abelian. When $q=3$, we know that the order of the ground walks is infinite by Lemma \ref{L:Order}. Thus, $\Pi(H(d,3))\cong \mathbb{Z}^d$. Lemma \ref{L:Order} then tells us that for $q\neq 3$ the order of the ground walks is 2. Thus, $\Pi(H(d,q))\cong \mathbb{Z}_2^d$.\\


\end{proof}

\begin{remark}\label{C:UC}
For $q=3$, we have that $|\Pi(H(d,3))|=\infty$, and so by Proposition \ref{P:index}, that each vertex $v\in H(d,3)$ lifts to an infinite number of possible points in $U$, its universal cover, and $U$ is infinite.  Also since $\Pi(H(d,3))$ has infinitely many subgroups of different indices, $H(d,3)$ has infinitely many covers (up to isomorphism).

On the other hand $v\in H(d,q)$ for $q>3$ lifts to $2^d$ vertices in its universal cover $U$, and thus $|V(U)|=q^d\cdot 2^d$ and $U$ is finite.
\end{remark}

\begin{corollary}
    
We now provide examples of covers in $H(2,3)$ where we have ground walks $U_1$ and $U_2$. 

\end{corollary}

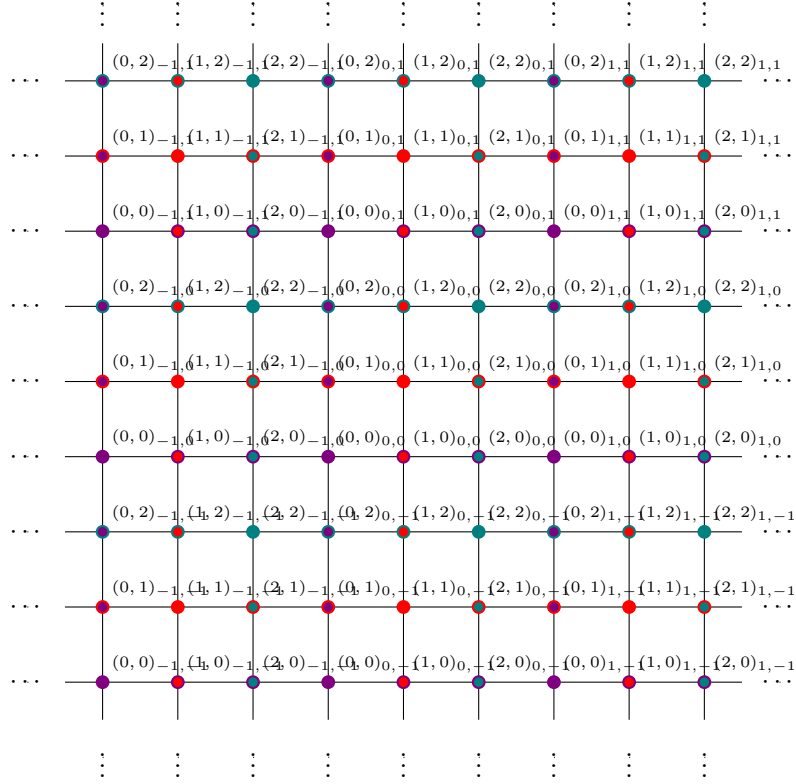
\begin{example}\label{E: universal}

\[\begin{tikzpicture}[scale=1]

\foreach \i in {-1,...,1}{
\foreach \j in {-1,...,1}{
\foreach \x in {0,...,2}{
\foreach \y in {0,...,2}{

\draw (\x+3*\j,\y+3*\i) --node[above right]{\tiny $(\x, \y)_{\j,\i}$} (\x+3*\j, \y+3*\i);
}
}
}
};

\foreach \x in {-3,...,5}{
\draw (\x, -3.5)--(\x, 5.5);
};

\foreach \y in {-3,...,5}{
\draw (-3.5, \y)--(5.5, \y);
};

\foreach \x in {-3,...,5}{
\draw (\x,-4) --node{$\vdots$} (\x, -4);
\draw (\x,6) --node{$\vdots$} (\x, 6);
};

\foreach \y in {-3,...,5}{
\draw (-4,\y) --node{$\cdots$} (-4, \y);
\draw (6,\y) --node{$\cdots$} (6, \y);
};

\draw[fill, violet] (0-3,0-3) circle (2pt);
\draw[fill, red] (1-3,0-3) circle (2pt);
\draw[fill, teal] (2-3,0-3) circle (2pt);
\draw[fill, violet] (0-3,1-3) circle (2pt);
\draw[fill, red] (1-3,1-3) circle (2pt);
\draw[fill, teal] (2-3,1-3) circle (2pt);
\draw[fill, violet] (0-3,2-3) circle (2pt);
\draw[fill, red] (1-3,2-3) circle (2pt);
\draw[fill, teal] (2-3,2-3) circle (2pt);

\draw[ thick, violet] (0-3,0-3) circle (2.2pt);
\draw[thick,  red] (0-3,1-3) circle (2.2pt);
\draw[ thick,  teal] (0-3,2-3) circle (2.2pt);
\draw[thick,  violet] (1-3,0-3) circle (2.2pt);
\draw[thick,  red] (1-3,1-3) circle (2.2pt);
\draw[thick,  teal] (1-3,2-3) circle (2.2pt);
\draw[thick,  violet] (2-3,0-3) circle (2.2pt);
\draw[thick,  red] (2-3,1-3) circle (2.2pt);
\draw[thick,  teal] (2-3,2-3) circle (2.2pt);

\draw[fill, violet] (0-3,0+3) circle (2pt);
\draw[fill, red] (1-3,0+3) circle (2pt);
\draw[fill, teal] (2-3,0+3) circle (2pt);
\draw[fill, violet] (0-3,1+3) circle (2pt);
\draw[fill, red] (1-3,1+3) circle (2pt);
\draw[fill, teal] (2-3,1+3) circle (2pt);
\draw[fill, violet] (0-3,2+3) circle (2pt);
\draw[fill, red] (1-3,2+3) circle (2pt);
\draw[fill, teal] (2-3,2+3) circle (2pt);

\draw[ thick, violet] (0-3,0+3) circle (2.2pt);
\draw[thick,  red] (0-3,1+3) circle (2.2pt);
\draw[ thick,  teal] (0-3,2+3) circle (2.2pt);
\draw[thick,  violet] (1-3,0+3) circle (2.2pt);
\draw[thick,  red] (1-3,1+3) circle (2.2pt);
\draw[thick,  teal] (1-3,2+3) circle (2.2pt);
\draw[thick,  violet] (2-3,0+3) circle (2.2pt);
\draw[thick,  red] (2-3,1+3) circle (2.2pt);
\draw[thick,  teal] (2-3,2+3) circle (2.2pt);

\draw[fill, violet] (0,0-3) circle (2pt);
\draw[fill, red] (1,0-3) circle (2pt);
\draw[fill, teal] (2,0-3) circle (2pt);
\draw[fill, violet] (0,1-3) circle (2pt);
\draw[fill, red] (1,1-3) circle (2pt);
\draw[fill, teal] (2,1-3) circle (2pt);
\draw[fill, violet] (0,2-3) circle (2pt);
\draw[fill, red] (1,2-3) circle (2pt);
\draw[fill, teal] (2,2-3) circle (2pt);

\draw[ thick, violet] (0,0-3) circle (2.2pt);
\draw[thick,  red] (0,1-3) circle (2.2pt);
\draw[ thick,  teal] (0,2-3) circle (2.2pt);
\draw[thick,  violet] (1,0-3) circle (2.2pt);
\draw[thick,  red] (1,1-3) circle (2.2pt);
\draw[thick,  teal] (1,2-3) circle (2.2pt);
\draw[thick,  violet] (2,0-3) circle (2.2pt);
\draw[thick,  red] (2,1-3) circle (2.2pt);
\draw[thick,  teal] (2,2-3) circle (2.2pt);

\draw[fill, violet] (0-3,0) circle (2pt);
\draw[fill, red] (1-3,0) circle (2pt);
\draw[fill, teal] (2-3,0) circle (2pt);
\draw[fill, violet] (0-3,1) circle (2pt);
\draw[fill, red] (1-3,1) circle (2pt);
\draw[fill, teal] (2-3,1) circle (2pt);
\draw[fill, violet] (0-3,2) circle (2pt);
\draw[fill, red] (1-3,2) circle (2pt);
\draw[fill, teal] (2-3,2) circle (2pt);

\draw[ thick, violet] (0-3,0) circle (2.2pt);
\draw[thick,  red] (0-3,1) circle (2.2pt);
\draw[ thick,  teal] (0-3,2) circle (2.2pt);
\draw[thick,  violet] (1-3,0) circle (2.2pt);
\draw[thick,  red] (1-3,1) circle (2.2pt);
\draw[thick,  teal] (1-3,2) circle (2.2pt);
\draw[thick,  violet] (2-3,0) circle (2.2pt);
\draw[thick,  red] (2-3,1) circle (2.2pt);
\draw[thick,  teal] (2-3,2) circle (2.2pt);

\draw[fill, violet] (0,0) circle (2pt);
\draw[fill, red] (1,0) circle (2pt);
\draw[fill, teal] (2,0) circle (2pt);
\draw[fill, violet] (0,1) circle (2pt);
\draw[fill, red] (1,1) circle (2pt);
\draw[fill, teal] (2,1) circle (2pt);
\draw[fill, violet] (0,2) circle (2pt);
\draw[fill, red] (1,2) circle (2pt);
\draw[fill, teal] (2,2) circle (2pt);

\draw[ thick, violet] (0,0) circle (2.2pt);
\draw[thick,  red] (0,1) circle (2.2pt);
\draw[ thick,  teal] (0,2) circle (2.2pt);
\draw[thick,  violet] (1,0) circle (2.2pt);
\draw[thick,  red] (1,1) circle (2.2pt);
\draw[thick,  teal] (1,2) circle (2.2pt);
\draw[thick,  violet] (2,0) circle (2.2pt);
\draw[thick,  red] (2,1) circle (2.2pt);
\draw[thick,  teal] (2,2) circle (2.2pt);

\draw[fill, violet] (0,0+3) circle (2pt);
\draw[fill, red] (1,0+3) circle (2pt);
\draw[fill, teal] (2,0+3) circle (2pt);
\draw[fill, violet] (0,1+3) circle (2pt);
\draw[fill, red] (1,1+3) circle (2pt);
\draw[fill, teal] (2,1+3) circle (2pt);
\draw[fill, violet] (0,2+3) circle (2pt);
\draw[fill, red] (1,2+3) circle (2pt);
\draw[fill, teal] (2,2+3) circle (2pt);

\draw[ thick, violet] (0,0+3) circle (2.2pt);
\draw[thick,  red] (0,1+3) circle (2.2pt);
\draw[ thick,  teal] (0,2+3) circle (2.2pt);
\draw[thick,  violet] (1,0+3) circle (2.2pt);
\draw[thick,  red] (1,1+3) circle (2.2pt);
\draw[thick,  teal] (1,2+3) circle (2.2pt);
\draw[thick,  violet] (2,0+3) circle (2.2pt);
\draw[thick,  red] (2,1+3) circle (2.2pt);
\draw[thick,  teal] (2,2+3) circle (2.2pt);

\draw[fill, violet] (0+3,0) circle (2pt);
\draw[fill, red] (1+3,0) circle (2pt);
\draw[fill, teal] (2+3,0) circle (2pt);
\draw[fill, violet] (0+3,1) circle (2pt);
\draw[fill, red] (1+3,1) circle (2pt);
\draw[fill, teal] (2+3,1) circle (2pt);
\draw[fill, violet] (0+3,2) circle (2pt);
\draw[fill, red] (1+3,2) circle (2pt);
\draw[fill, teal] (2+3,2) circle (2pt);

\draw[ thick, violet] (0+3,0) circle (2.2pt);
\draw[thick,  red] (0+3,1) circle (2.2pt);
\draw[ thick,  teal] (0+3,2) circle (2.2pt);
\draw[thick,  violet] (1+3,0) circle (2.2pt);
\draw[thick,  red] (1+3,1) circle (2.2pt);
\draw[thick,  teal] (1+3,2) circle (2.2pt);
\draw[thick,  violet] (2+3,0) circle (2.2pt);
\draw[thick,  red] (2+3,1) circle (2.2pt);
\draw[thick,  teal] (2+3,2) circle (2.2pt);

\draw[fill, violet] (0+3,0+3) circle (2pt);
\draw[fill, red] (1+3,0+3) circle (2pt);
\draw[fill, teal] (2+3,0+3) circle (2pt);
\draw[fill, violet] (0+3,1+3) circle (2pt);
\draw[fill, red] (1+3,1+3) circle (2pt);
\draw[fill, teal] (2+3,1+3) circle (2pt);
\draw[fill, violet] (0+3,2+3) circle (2pt);
\draw[fill, red] (1+3,2+3) circle (2pt);
\draw[fill, teal] (2+3,2+3) circle (2pt);

\draw[ thick, violet] (0+3,0+3) circle (2.2pt);
\draw[thick,  red] (0+3,1+3) circle (2.2pt);
\draw[ thick,  teal] (0+3,2+3) circle (2.2pt);
\draw[thick,  violet] (1+3,0+3) circle (2.2pt);
\draw[thick,  red] (1+3,1+3) circle (2.2pt);
\draw[thick,  teal] (1+3,2+3) circle (2.2pt);
\draw[thick,  violet] (2+3,0+3) circle (2.2pt);
\draw[thick,  red] (2+3,1+3) circle (2.2pt);
\draw[thick,  teal] (2+3,2+3) circle (2.2pt);

\draw[fill, violet] (0+3,0-3) circle (2pt);
\draw[fill, red] (1+3,0-3) circle (2pt);
\draw[fill, teal] (2+3,0-3) circle (2pt);
\draw[fill, violet] (0+3,1-3) circle (2pt);
\draw[fill, red] (1+3,1-3) circle (2pt);
\draw[fill, teal] (2+3,1-3) circle (2pt);
\draw[fill, violet] (0+3,2-3) circle (2pt);
\draw[fill, red] (1+3,2-3) circle (2pt);
\draw[fill, teal] (2+3,2-3) circle (2pt);

\draw[ thick, violet] (0+3,0-3) circle (2.2pt);
\draw[thick,  red] (0+3,1-3) circle (2.2pt);
\draw[ thick,  teal] (0+3,2-3) circle (2.2pt);
\draw[thick,  violet] (1+3,0-3) circle (2.2pt);
\draw[thick,  red] (1+3,1-3) circle (2.2pt);
\draw[thick,  teal] (1+3,2-3) circle (2.2pt);
\draw[thick,  violet] (2+3,0-3) circle (2.2pt);
\draw[thick,  red] (2+3,1-3) circle (2.2pt);
\draw[thick,  teal] (2+3,2-3) circle (2.2pt);

\end{tikzpicture}\]

We begin with the universal cover of $H(2,3)$ depicted above, where vertices are $v_{ij}, v\in V(H_2,3),$ $ i, j\in \mathbb{Z}$ and incidence is determined by $(2,x)_{ij} \sim (0,x)_{(i+1)j}$ and with covering map $\rho:U\to H(2,3), $ $ (x,y)_{ij}\mapsto (x,y)$. Comparing this picture to Example \ref{H(2,3)}, we see that if we lift $(0,0)$ to say $(0,0)_{00}$, then any walk in $H(2,3)$ lifts to a walk in $U$.  Moreover, every spider move of a walk lifts to a spider move in $U$ as well since spider moves correspond to switching the order of a first and second coordinate change \cite{CS3}.  Moreover, a walk $W$ in $H(2,3)$ starting at $(0,0)$ lifts to a walk starting and ending at $(0,0)_{0,0}$ if and only if $W$ is equivalent to the trivial walk, since ground walks $U_1, U_2$ lift to moving to the right or up by 3 places respectively. Since $|\Pi(H(d,3))|=\infty$, the universal cover is also infinite, so repeating vertices in $W$ will lift to unique vertices in $U$.

\end{example}

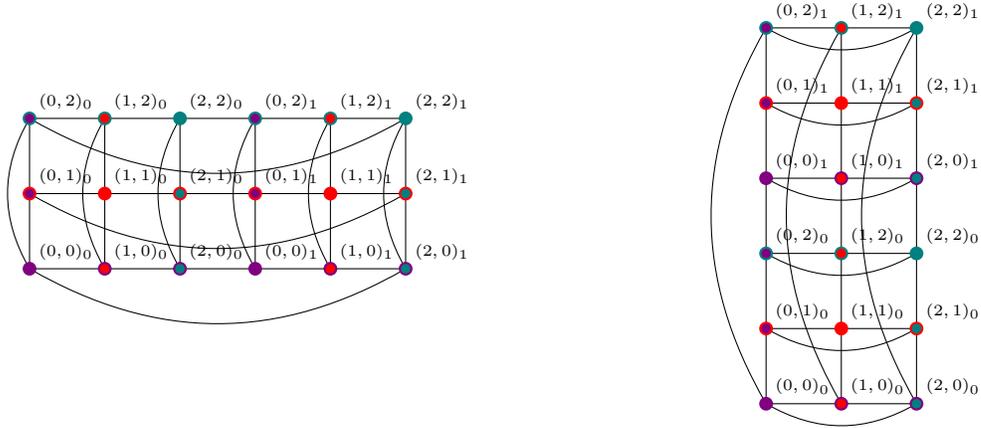
\begin{example}\label{E:2 horizontal}

\[\begin{tikzpicture}
    \node (A) at (-5,0){\begin{tikzpicture}[scale=1]

\foreach \i in {0,...,1}{
\foreach \x in {0,...,2}{
\foreach \y in {0,...,2}{

\draw (\x+3*\i,\y) --node[above right]{\tiny $(\x, \y)_{\i}$} (\x+3*\i, \y);

}
}
};

\draw (0,0)--(5,0);
\draw (0,2)--(5,2);
\draw (0,1)--(5,1);

\foreach \x in {0,...,5}{
\draw (\x, 0)--(\x, 2);
\draw (\x,0)edge[bend left](\x,2);
};

\foreach \y in {0,...,2}{
\draw (0,\y)edge[bend right](5,\y);

};

\draw[fill, violet] (0,0) circle (2pt);
\draw[fill, red] (1,0) circle (2pt);
\draw[fill, teal] (2,0) circle (2pt);
\draw[fill, violet] (0,1) circle (2pt);
\draw[fill, red] (1,1) circle (2pt);
\draw[fill, teal] (2,1) circle (2pt);
\draw[fill, violet] (0,2) circle (2pt);
\draw[fill, red] (1,2) circle (2pt);
\draw[fill, teal] (2,2) circle (2pt);

\draw[ thick, violet] (0,0) circle (2.2pt);
\draw[thick,  red] (0,1) circle (2.2pt);
\draw[ thick,  teal] (0,2) circle (2.2pt);
\draw[thick,  violet] (1,0) circle (2.2pt);
\draw[thick,  red] (1,1) circle (2.2pt);
\draw[thick,  teal] (1,2) circle (2.2pt);
\draw[thick,  violet] (2,0) circle (2.2pt);
\draw[thick,  red] (2,1) circle (2.2pt);
\draw[thick,  teal] (2,2) circle (2.2pt);


\draw[fill, violet] (0+3,0) circle (2pt);
\draw[fill, red] (1+3,0) circle (2pt);
\draw[fill, teal] (2+3,0) circle (2pt);
\draw[fill, violet] (0+3,1) circle (2pt);
\draw[fill, red] (1+3,1) circle (2pt);
\draw[fill, teal] (2+3,1) circle (2pt);
\draw[fill, violet] (0+3,2) circle (2pt);
\draw[fill, red] (1+3,2) circle (2pt);
\draw[fill, teal] (2+3,2) circle (2pt);

\draw[ thick, violet] (0+3,0) circle (2.2pt);
\draw[thick,  red] (0+3,1) circle (2.2pt);
\draw[ thick,  teal] (0+3,2) circle (2.2pt);
\draw[thick,  violet] (1+3,0) circle (2.2pt);
\draw[thick,  red] (1+3,1) circle (2.2pt);
\draw[thick,  teal] (1+3,2) circle (2.2pt);
\draw[thick,  violet] (2+3,0) circle (2.2pt);
\draw[thick,  red] (2+3,1) circle (2.2pt);
\draw[thick,  teal] (2+3,2) circle (2.2pt);

\end{tikzpicture}};

\node (B) at (3,0){\begin{tikzpicture}[scale=1]

\foreach \i in {0,...,1}{
\foreach \x in {0,...,2}{
\foreach \y in {0,...,2}{

\draw (\x,\y+3*\i) --node[above right]{\tiny $(\x, \y)_{\i}$} (\x, \y+3*\i);

}
}
};

\draw (0,0)--(0,5);
\draw (1,0)--(1,5);
\draw (2,0)--(2,5);

\foreach \y in {0,...,5}{
\draw (0, \y)--(2, \y);
\draw (0,\y)edge[bend right](2,\y);
};

\foreach \x in {0,...,2}{
\draw (\x,0)edge[bend left](\x,5);
};

\draw[fill, violet] (0,0) circle (2pt);
\draw[fill, red] (1,0) circle (2pt);
\draw[fill, teal] (2,0) circle (2pt);
\draw[fill, violet] (0,1) circle (2pt);
\draw[fill, red] (1,1) circle (2pt);
\draw[fill, teal] (2,1) circle (2pt);
\draw[fill, violet] (0,2) circle (2pt);
\draw[fill, red] (1,2) circle (2pt);
\draw[fill, teal] (2,2) circle (2pt);

\draw[ thick, violet] (0,0) circle (2.2pt);
\draw[thick,  red] (0,1) circle (2.2pt);
\draw[ thick,  teal] (0,2) circle (2.2pt);
\draw[thick,  violet] (1,0) circle (2.2pt);
\draw[thick,  red] (1,1) circle (2.2pt);
\draw[thick,  teal] (1,2) circle (2.2pt);
\draw[thick,  violet] (2,0) circle (2.2pt);
\draw[thick,  red] (2,1) circle (2.2pt);
\draw[thick,  teal] (2,2) circle (2.2pt);

\draw[fill, violet] (0,0+3) circle (2pt);
\draw[fill, red] (1,0+3) circle (2pt);
\draw[fill, teal] (2,0+3) circle (2pt);
\draw[fill, violet] (0,1+3) circle (2pt);
\draw[fill, red] (1,1+3) circle (2pt);
\draw[fill, teal] (2,1+3) circle (2pt);
\draw[fill, violet] (0,2+3) circle (2pt);
\draw[fill, red] (1,2+3) circle (2pt);
\draw[fill, teal] (2,2+3) circle (2pt);

\draw[ thick, violet] (0,0+3) circle (2.2pt);
\draw[thick,  red] (0,1+3) circle (2.2pt);
\draw[ thick,  teal] (0,2+3) circle (2.2pt);
\draw[thick,  violet] (1,0+3) circle (2.2pt);
\draw[thick,  red] (1,1+3) circle (2.2pt);
\draw[thick,  teal] (1,2+3) circle (2.2pt);
\draw[thick,  violet] (2,0+3) circle (2.2pt);
\draw[thick,  red] (2,1+3) circle (2.2pt);
\draw[thick,  teal] (2,2+3) circle (2.2pt);

\end{tikzpicture}};

\end{tikzpicture}\]

Observe now the following covers of $H(2,3)$: $U/\langle U_1^2,U_2\rangle$ and $U/\langle U_1,U_2^2\rangle$. In the first cover, we see that every walk in $H(2,3)$ lifts to a walk in $U/\langle U_1^2,U_2\rangle$. Observe that in this cover a walk in $H(2,3)$ starting at $(0,0)$ lifts to a walk starting and ending at $(0,0)_{0,0}$ when $W$ is equivalent to the product of 2 horizontal ground walks and one vertical ground walk. The second cover follows similarly. In this cover, a walk $W$ in $H(2,3)$ starting at $(0,0)$ lifts to a walk starting and ending at $(0,0)_{0,0}$ when $W$ is equivalent to the product of 2 vertical ground walks and 1 horizontal ground walk.

\end{example}

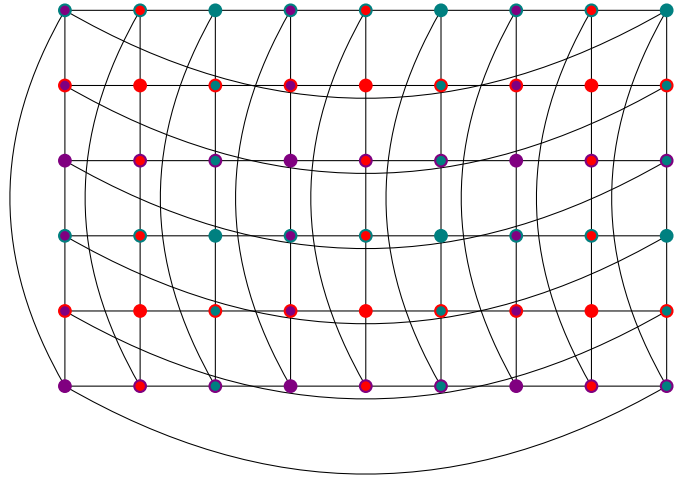
\begin{example}\label{E:3 horiz 2 vert}

\[\begin{tikzpicture}[scale=1]

\foreach \i in {0,...,1}{
\foreach \j in {-1,...,1}{
\foreach \x in {0,...,2}{
\foreach \y in {0,...,2}{

}
}
}
};

\foreach \x in {-3,...,5}{
\draw (\x, 0)--(\x, 5);
};

\foreach \y in {0,...,5}{
\draw (-3, \y)--(5, \y);
};

\foreach \x in {-3,...,5}{
\draw (\x,0)edge[bend left](\x,5);

};

\foreach \y in {0,...,5}{
\draw (-3,\y)edge[bend right](5,\y);

};

\draw[fill, violet] (0-3,0) circle (2pt);
\draw[fill, red] (1-3,0) circle (2pt);
\draw[fill, teal] (2-3,0) circle (2pt);
\draw[fill, violet] (0-3,1) circle (2pt);
\draw[fill, red] (1-3,1) circle (2pt);
\draw[fill, teal] (2-3,1) circle (2pt);
\draw[fill, violet] (0-3,2) circle (2pt);
\draw[fill, red] (1-3,2) circle (2pt);
\draw[fill, teal] (2-3,2) circle (2pt);

\draw[ thick, violet] (0-3,0) circle (2.2pt);
\draw[thick,  red] (0-3,1) circle (2.2pt);
\draw[ thick,  teal] (0-3,2) circle (2.2pt);
\draw[thick,  violet] (1-3,0) circle (2.2pt);
\draw[thick,  red] (1-3,1) circle (2.2pt);
\draw[thick,  teal] (1-3,2) circle (2.2pt);
\draw[thick,  violet] (2-3,0) circle (2.2pt);
\draw[thick,  red] (2-3,1) circle (2.2pt);
\draw[thick,  teal] (2-3,2) circle (2.2pt);

\draw[fill, violet] (0-3,0+3) circle (2pt);
\draw[fill, red] (1-3,0+3) circle (2pt);
\draw[fill, teal] (2-3,0+3) circle (2pt);
\draw[fill, violet] (0-3,1+3) circle (2pt);
\draw[fill, red] (1-3,1+3) circle (2pt);
\draw[fill, teal] (2-3,1+3) circle (2pt);
\draw[fill, violet] (0-3,2+3) circle (2pt);
\draw[fill, red] (1-3,2+3) circle (2pt);
\draw[fill, teal] (2-3,2+3) circle (2pt);

\draw[ thick, violet] (0-3,0+3) circle (2.2pt);
\draw[thick,  red] (0-3,1+3) circle (2.2pt);
\draw[ thick,  teal] (0-3,2+3) circle (2.2pt);
\draw[thick,  violet] (1-3,0+3) circle (2.2pt);
\draw[thick,  red] (1-3,1+3) circle (2.2pt);
\draw[thick,  teal] (1-3,2+3) circle (2.2pt);
\draw[thick,  violet] (2-3,0+3) circle (2.2pt);
\draw[thick,  red] (2-3,1+3) circle (2.2pt);
\draw[thick,  teal] (2-3,2+3) circle (2.2pt);

\draw[fill, violet] (0,0) circle (2pt);
\draw[fill, red] (1,0) circle (2pt);
\draw[fill, teal] (2,0) circle (2pt);
\draw[fill, violet] (0,1) circle (2pt);
\draw[fill, red] (1,1) circle (2pt);
\draw[fill, teal] (2,1) circle (2pt);
\draw[fill, violet] (0,2) circle (2pt);
\draw[fill, red] (1,2) circle (2pt);
\draw[fill, teal] (2,2) circle (2pt);

\draw[ thick, violet] (0,0) circle (2.2pt);
\draw[thick,  red] (0,1) circle (2.2pt);
\draw[ thick,  teal] (0,2) circle (2.2pt);
\draw[thick,  violet] (1,0) circle (2.2pt);
\draw[thick,  red] (1,1) circle (2.2pt);
\draw[thick,  teal] (1,2) circle (2.2pt);
\draw[thick,  violet] (2,0) circle (2.2pt);
\draw[thick,  red] (2,1) circle (2.2pt);
\draw[thick,  teal] (2,2) circle (2.2pt);

\draw[fill, violet] (0,0+3) circle (2pt);
\draw[fill, red] (1,0+3) circle (2pt);
\draw[fill, teal] (2,0+3) circle (2pt);
\draw[fill, violet] (0,1+3) circle (2pt);
\draw[fill, red] (1,1+3) circle (2pt);
\draw[fill, teal] (2,1+3) circle (2pt);
\draw[fill, violet] (0,2+3) circle (2pt);
\draw[fill, red] (1,2+3) circle (2pt);
\draw[fill, teal] (2,2+3) circle (2pt);

\draw[ thick, violet] (0,0+3) circle (2.2pt);
\draw[thick,  red] (0,1+3) circle (2.2pt);
\draw[ thick,  teal] (0,2+3) circle (2.2pt);
\draw[thick,  violet] (1,0+3) circle (2.2pt);
\draw[thick,  red] (1,1+3) circle (2.2pt);
\draw[thick,  teal] (1,2+3) circle (2.2pt);
\draw[thick,  violet] (2,0+3) circle (2.2pt);
\draw[thick,  red] (2,1+3) circle (2.2pt);
\draw[thick,  teal] (2,2+3) circle (2.2pt);

\draw[fill, violet] (0+3,0) circle (2pt);
\draw[fill, red] (1+3,0) circle (2pt);
\draw[fill, teal] (2+3,0) circle (2pt);
\draw[fill, violet] (0+3,1) circle (2pt);
\draw[fill, red] (1+3,1) circle (2pt);
\draw[fill, teal] (2+3,1) circle (2pt);
\draw[fill, violet] (0+3,2) circle (2pt);
\draw[fill, red] (1+3,2) circle (2pt);
\draw[fill, teal] (2+3,2) circle (2pt);

\draw[ thick, violet] (0+3,0) circle (2.2pt);
\draw[thick,  red] (0+3,1) circle (2.2pt);
\draw[ thick,  teal] (0+3,2) circle (2.2pt);
\draw[thick,  violet] (1+3,0) circle (2.2pt);
\draw[thick,  red] (1+3,1) circle (2.2pt);
\draw[thick,  teal] (1+3,2) circle (2.2pt);
\draw[thick,  violet] (2+3,0) circle (2.2pt);
\draw[thick,  red] (2+3,1) circle (2.2pt);
\draw[thick,  teal] (2+3,2) circle (2.2pt);

\draw[fill, violet] (0+3,0+3) circle (2pt);
\draw[fill, red] (1+3,0+3) circle (2pt);
\draw[fill, teal] (2+3,0+3) circle (2pt);
\draw[fill, violet] (0+3,1+3) circle (2pt);
\draw[fill, red] (1+3,1+3) circle (2pt);
\draw[fill, teal] (2+3,1+3) circle (2pt);
\draw[fill, violet] (0+3,2+3) circle (2pt);
\draw[fill, red] (1+3,2+3) circle (2pt);
\draw[fill, teal] (2+3,2+3) circle (2pt);

\draw[ thick, violet] (0+3,0+3) circle (2.2pt);
\draw[thick,  red] (0+3,1+3) circle (2.2pt);
\draw[ thick,  teal] (0+3,2+3) circle (2.2pt);
\draw[thick,  violet] (1+3,0+3) circle (2.2pt);
\draw[thick,  red] (1+3,1+3) circle (2.2pt);
\draw[thick,  teal] (1+3,2+3) circle (2.2pt);
\draw[thick,  violet] (2+3,0+3) circle (2.2pt);
\draw[thick,  red] (2+3,1+3) circle (2.2pt);
\draw[thick,  teal] (2+3,2+3) circle (2.2pt);

\end{tikzpicture}\]

In this last cover we have $U/\langle U_1^3,U_2^2\rangle$. In this cover, a walk $W$ in $H(2,3)$ starting at $(0,0)$ lifts to a walk starting and ending at $(0,0)_{0,0}$ when $W$ is equivalent to the product of 2 vertical ground walks and 3 horizontal ground walks.

\end{example}

\begin{example}\label{E:UCH24}
Consider $H(2,4)$:

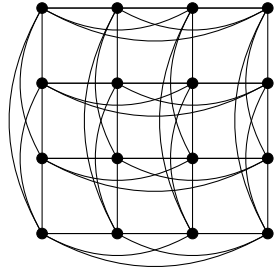
\[\begin{tikzpicture}

\foreach \x in {0,...,3}{
    \foreach \y in {0,...,3}{

    };
};

\foreach \x in {0,...,3}{
\draw (\x, 0)--(\x, 3);
\draw (\x,0)edge[bend left](\x,2);
\draw (\x,1)edge[bend left](\x,3);
\draw (\x,0)edge[bend left](\x,3);

};

\draw (0,0)--(3,0);
\draw (0,2)--(3,2);
\draw (0,1)--(3,1);
\draw (0,3)--(3,3);

\foreach \y in {0,...,3}{
\draw (0, \y)--(3, \y);
\draw (0,\y)edge[bend right](2,\y);
\draw (1,\y)edge[bend right](3,\y);
\draw (0,\y)edge[bend right](3,\y);

};

\draw[fill, black] (0,0) circle (2pt);
\draw[fill, black] (1,0) circle (2pt);
\draw[fill, black] (2,0) circle (2pt);
\draw[fill, black] (0,1) circle (2pt);
\draw[fill, black] (1,1) circle (2pt);
\draw[fill, black] (2,1) circle (2pt);
\draw[fill, black] (0,2) circle (2pt);
\draw[fill, black] (1,2) circle (2pt);
\draw[fill, black] (2,2) circle (2pt);
\draw[fill, black] (3,0) circle (2pt);
\draw[fill, black] (3,1) circle (2pt);
\draw[fill, black] (3,2) circle (2pt);
\draw[fill, black] (0,3) circle (2pt);
\draw[fill, black] (1,3) circle (2pt);
\draw[fill, black] (2,3) circle (2pt);
\draw[fill, black] (3,3) circle (2pt);

 \end{tikzpicture}\]
 
The universal cover of $H(2,4)$ is finite since $\Pi(H(2,4))\cong \mathbb{Z}_2\times \mathbb{Z}_2$.  Since the $|\Pi(H(2,4))|=4$, the universal cover of $H(2,4)$ is the following 4-cover:

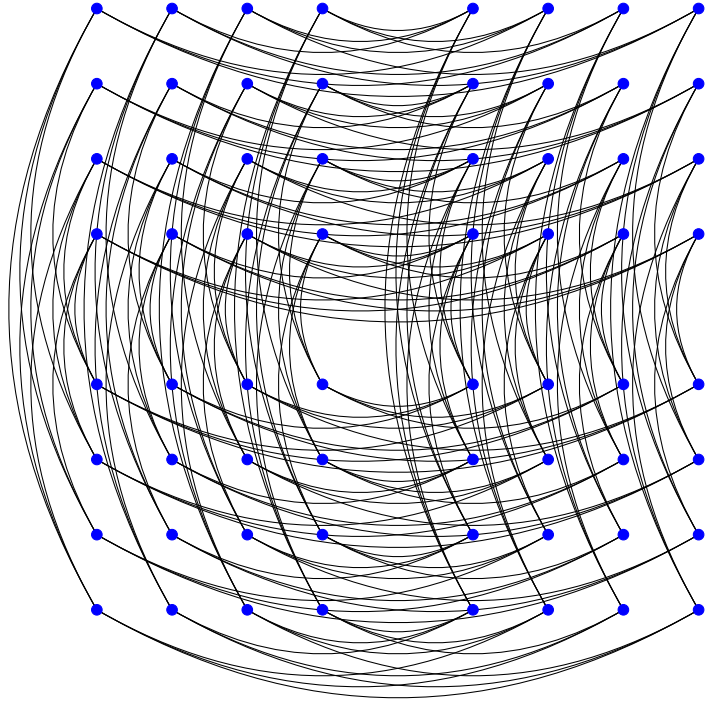
\[
\begin{tikzpicture}

    \foreach \x in {0,...,3}{
            \foreach \y in {0,...,3}{
            \foreach \i in {0,1}{
                \draw (\x,\y+5*\i) edge[bend right] ({Mod(\x+1,4)+5} , \y+5*\i);
                \draw (\x,\y+5*\i) edge[bend right] ({Mod(\x+2,4)+5} , \y+5*\i);
                \draw (\x,\y+5*\i) edge[bend right] ({Mod(\x+3,4)+5} , \y+5*\i);

                \draw (\x+5*\i,\y) edge[bend left] (\x+5*\i, {Mod(\y+1,4)+5} );
                \draw (\x+5*\i,\y) edge[bend left] (\x+5*\i, {Mod(\y+2,4)+5});
                \draw (\x+5*\i,\y) edge[bend left] (\x+5*\i, {Mod(\y+3,4)+5} );

                }
        }
            
    }

    \foreach \i in {0,...,1}{
        \foreach \j in {0,...,1}{
        
            \foreach \x in {0,...,3}{
                \foreach \y in {0,...,3}{
                \draw[fill, blue] (\x+5*\i,\y+5*\j) 
 circle (2pt);

                }
            
            }
        }
    }

\end{tikzpicture}
\]

    We can see that either ground walk $U_1, U_2$ would lift to a walk that starts at $(0,0)_{0,0}$ that ends at $(0,0)_{1,0}$ or $(0,0)_{0,1}$, and that $U_i^2$ lifts to a walk ending at $(0,0)_{0,0}$.

\end{example}

\section{Conclusions and Future Directions}\label{S:Conclusion}

To conclude, we found that the fundamental groups of Hamming graphs $H(q,d)$ are trivial when $q=2$, $\mathbb{Z}^d$ when $q=3$ and $\mathbb{Z}_2^d$ when $q>3$. We also showed that all covers of Hamming graphs are generated by ground walks, which are walks of length 3 with non-zero coordinate $i$, and these generators commute and are unique. Thus, Hamming graphs are cartesian or box products of complete graphs. \\

In this research, we computed the fundamental group for a more complex family of graphs than those previously studied under this x-homotopy. In continuing this work, we would like to extend this investigation to include more families of graphs with similar or greater complexity. In a work in progress, Chih and Scull generalized some of these results to show that $\Pi(G\ \square\ H)\cong \Pi(G)\times \Pi(H)$.\\

\section{Acknowledgments}

We wish to thank Dr.\ Laura Scull for their comments and insight, Dr.\ James Nagy for being supportive of this project, the Mathematical Association of America for recognizing this work at MathFest 2023, and the referee for their helpful comments.


\bibliographystyle{spmpsci}      
\bibliography{ref}   

\end{document}